\documentclass{article}
\usepackage{lipsum}
\usepackage{graphicx,bm}
\usepackage{IEEEtrantools}
\usepackage{latexsym}
\usepackage{amsfonts}
\usepackage{amssymb}
\usepackage{amsmath,amsfonts,amsthm}
\usepackage{mathrsfs}
\usepackage[all]{xy}
\usepackage{epic}
\usepackage{eepic}
\usepackage{enumerate}
\usepackage{multirow}
\usepackage[breaklinks]{hyperref}
\usepackage{tikz}

\usepackage{mathrsfs}
\usepackage[all]{xy}
\usepackage{epic}
\usepackage{eepic}
\usepackage{enumerate}
\usepackage{multirow}
\usepackage[breaklinks]{hyperref}
\usepackage{tikz}
\usetikzlibrary{matrix,arrows,decorations.pathmorphing}
\usepackage{color}
\usepackage[breaklinks]{hyperref}
\usepackage{tikz-cd}

\usepackage{mathrsfs}
\usepackage[all]{xy}
\usepackage{epic}
\usepackage{eepic}
\usepackage{enumerate}
\usepackage{multirow}
\usepackage[breaklinks]{hyperref}
\usepackage{tikz}
\usetikzlibrary{arrows,decorations.markings}
\usepackage{color}

\newcommand{\bbbt}{\mathbb{T}}

\newcommand{\be}{\begin{equation}}
	\newcommand{\ee}{\end{equation}}
\newcommand{\bea}{\begin{eqnarray}}
	\newcommand{\eea}{\end{eqnarray}}
\newcommand{\bean}{\begin{eqnarray*}}
	\newcommand{\eean}{\end{eqnarray*}}
\newcommand{\brray}{\begin{array}}
	\newcommand{\erray}{\end{array}}
\newcommand{\biearray}{\begin{IEEEarray}{rCl}}
	\newcommand{\eiearray}{\end{IEEEarray}}
\newcommand{\newsection}[1]{\setcounter{equation}{0}
	\setcounter{definition}{0}
	\section{#1}}

\newtheorem{definition}{Definition}[section]
\newtheorem{theorem}[definition]{Theorem}
\newtheorem{lemma}[definition]{Lemma}
\newtheorem{proposition}[definition]{Proposition}
\newtheorem{corollary}[definition]{Corollary}

\newtheorem{remark}[definition]{Remark}

\newcommand\restr[2]{{% we make the whole thing an ordinary symbol
		\left.\kern-\nulldelimiterspace % automatically resize the bar with \right
		#1 % the function
		\littletaller % pretend it's a little taller at normal size
		\right|_{#2} % this is the delimiter
}}
\newcommand{\littletaller}{\mathchoice{\vphantom{\big|}}{}{}{}}

\begin{document}
	%%%%%%%%%%%%%%%%%%%%%%%%%%%%%%%%%
	\author{{\sc Shreema Subhash Bhatt, Bipul Saurabh}}
	\title{ $K$-stability of $C^*$-algebras  generated by isometries and unitaries with twisted commutation relations}
	\maketitle
	%%%%%%%%%%%%%%%%%%%%%%%%%%%%%%%%%%
	%%%%%  ABSTRACT
	%%%%%%%%%%%%%%%%%%%%%%%%%%%%%%%%%%

%%author names are separated by comma (,)
%%use \and before the last author name
%%\textsuperscript{number} is used for affiliation
%%use a * along with the number separated by comma
%% for the  author for correspondence

%\author{Author1\textsuperscript{1}, Author2\textsuperscript{1}\and Author3\textsuperscript{2,*}}
%\affilOne{\textsuperscript{1} Department of P, University X\\}

\maketitle

%%include \msinfo for
%%manuscript information such as
%%received, revised dates
%%

\begin{abstract}
	In this article, we define a family of $C^*$-algebras that  are generated by a finite set of unitaries and isometries satisfying  certain twisted commutation relations and prove their $K$-stability. This family includes the $C^*$-algebra  of doubly non-commuting isometries and free twist of isometries.
	Next, we consider the $C^*$-algebra  $A_{\mathcal{V}}$ generated by an $n$-tuple of $\mathcal{U}$-twisted  isometries  $\mathcal{V}$ with respect to a fixed  $n\choose 2$-tuple  $\mathcal{U}=\{U_{ij}:1\leq i<j \leq n\}$ of commuting unitaries (see \cite{NarJaySur-2022aa}).
	Identifying any point of the joint spectrum $\sigma(\mathcal{U})$ of the commutative $C^{*}$-algebra generated by $(\{U_{ij}:1\leq i<j \leq n\})$ with a skew-symmetric matrix, we show that the algebra $A_{\mathcal{V}}$
	is $K$-stable under the assumption that $\sigma(\mathcal{U})$ does not contain any degenerate, skew-symmetric matrix.
	Finally, we prove the same result for the $C^*$-algebra  generated by a tuple of free $\mathcal{U}$-twisted  isometries.
\end{abstract}
\textbf{keywords}: Isometries; von Neumann-Wold decomposition; $K$-stability; quasi unitary; noncommutative torus.

\textbf{AMS Subject Class}: 46L85, 46L80, 47L55.

\setcounter{page}{1}

\section{Introduction}
Given a unital $C^*$-algebra $\mathcal{A}$, one can attach two abelian groups $K_0(\mathcal{A})$, and $K_1(\mathcal{A})$ to $\mathcal{A}$.
These invariants played a crucial role in the classification of  purely infinite simple separable amenable unital $C^*$-algebras
which satisfy the Universal Coefficient Theorem (UCT) (see \cite{Phi-2000aa}).   However, they  do not distinguish $\mathcal{A}$, and $\mathcal{K} \otimes \mathcal{A}$, where $\mathcal{K}$ is the $C^*$-algebra of compact operators. One therefore calls them stable $K$-groups.  On the other hand,
there are
homotopy groups $\pi_k(U_m(\mathcal{A}))$ of  the group $U_m(\mathcal{A})$ of $m \times m$ unitary matrices over $\mathcal{A}$, which  are collectively called   non-stable $K$-groups of $\mathcal{A}$.
There is a canonical  inclusion map $i_n$ from $\mathcal{A}$ to $M_n(\mathcal{A})$. By functoriality, this induces a map $(i_n)_*$ from $\pi_k(U_m(\mathcal{A}))$ to $\pi_k(U_m(M_n(\mathcal{A})))$. We say that a $C^*$-algebra $\mathcal{A}$ is $K$-stable if  $(i_n)_*$ is an isomorphism for all $n \in \mathbb{N}$. For a $K$-stable $C^*$-algebra $\mathcal{A}$, $\pi_k(U_m(\mathcal{A}))$ is canonically isomorphic to $K_0(\mathcal{A})$ for $k$ even, and for $k$ odd, it is the same as $K_1(\mathcal{A})$.   This property comes in handy in many situations where the direct computation of non-stable $K$-groups is difficult, which is often the case.  Even for the algebra of complex numbers, these groups are unknown to date. The first instance of computation of non-stable $K$-groups can be traced back to
Rieffel (\cite{Rie-1987aa}), who  computed these groups for irrational non-commutative $m$-torus $\mathcal{A}_{\theta}$,  and established  its $K$-stability.
By introducing the notion of a quasi unitary, Thomsen (\cite{Tho-1991aa}) extended the  non-stable $K$-theory to  nonunital $C^*$-algebras. He put the non stable $K$-theory under the framework of homology theory, which makes the computation of these groups more tractable. Under some mild assumptions on a locally compact Hausdorff space $X$, Seth and Vaidynathan (\cite{SetVai-2020aa}) showed that a continuous $C_0(X)$-algebra is $K$-stable if all of its fibers are $K$-stable.  %

Isometries play central roles in both Operator theory and Operator algebra.
The classical von-Neumann Wold Decomposition Theorem provides a fundamental understanding of an isometry  on separable Hilbert spaces. It states that up to unitarily equivalence, such  an isometry is either a shift, or a unitary, or a direct sum of a shift and a unitary. Using this, Coburn (\cite{Cob-1967aa}) classified all $C^*$-algebras which are generated by an isometry.
Later,  Berger, Coburn, and Lebow  (\cite{BerCobLeb-1978aa}) studied  the representation theory of the $C^*$-algebra generated by a commuting  tuple isometries acting on a separable Hilbert space, and under some additional hypothesis, they identified all Fredholm operators in such a $C^*$-algebra.   Many $C^*$-algebras which are   generated by a tuple of isometries exhibiting certain twisted commutation relations have been investigated  (see \cite{Rie-1987aa}, \cite{Mor-2013aa}, \cite{MarPau-2016aa}, \cite{NarJaySur-2022aa} and references therein). In this paper,  we take up certain families of
$C^*$-algebras whose homomorphic avatars encompass all the examples studied in these articles. We discuss the general form of a representation of such noncommutative spaces,   explore their topological structure,  and  prove their $K$-stability. One of our main motivations  is to study nontrivial geometries of these spaces in the sense of Connes (\cite{Con-1994aa}), and this article is a step forward. To prove that a geometrical data (a spectral triple) is nontrivial, one needs to pair it with  $K$-groups, which we are trying to compute here. Moreover, $K$-stabilty allows us to take unitaries or projections in the algebra itself rather than going into matrix algebra, where pairing may be more difficult.

Here is a brief outline of the contents of this paper. To that end, it is necessary to first define certain families of $C^*$-algebras. Let $m,n \in \mathbb{Z}_+$ and let $\Theta=\{\theta_{ij}: 1\leq i<j \leq m+n\}$ be an $m+n\choose 2$-tuple of scalars. Let  $B_{\Theta}^{m,n}$ be the universal $C^*$-algebra generated by $m$ unitaries $s_1, s_2, \cdots s_m$, and $n$ isometries $s_{m+1}, \cdots  ,s_{m+n}$ such that  $s_is_j=e^{2\pi\iota \theta_{ij}}s_js_i,\mbox{ for } 1\leq i<j \leq m+n$. Similarly, one defines $C_{\Theta}^{m,n}$ by imposing a stronger commutation  relation, namely $s_i^*s_j=e^{-2\pi\iota\theta_{ij}}s_js_i^*, \mbox{ for } 1\leq i<j \leq m+n$.  In the next section, we review representation theory of $C_{\Theta}^{m,n}$  by invoking the von Neumann-Wold decomposition proved in \cite{NarJaySur-2022aa}. We show that $C_{\Theta}^{m,n}$ has a unique nontrivial minimal ideal.  Moreover, we embed  $C_{\Theta}^{m,n}$ faithfully  in the $C^*$-algebra of bounded linear operators acting on a Hilbert space.
In section 3, we produce a chain of short exact sequences of $C^*$-algebras associated to $C_{\Theta}^{m,n}$, and  compute $K$-groups of $C_{\Theta}^{m,n}$ with explicit generators.
Employing the Five  lemma of homology theory and a result of \cite{Tho-1991aa}, we  prove  $K$-stability of any closed ideal and any homomorphic image of $C_{\Theta}^{m,n}$.  As a consequence, we get  their non-stable $K$-groups as well.  Next, we prove that they are in the bootstrap category $\mathcal{N}$, hence satisfy the Universal Co-efficient Theorem (abbreviated as UCT).

Section $4$ can be looked upon as the heart of the present paper. It deals with a family of   more complicated $C^*$-algebras, namely $B_{\Theta}^{m,n}$.  These $C^*$-algebras are not even exact.   Weber \cite{Mor-2013aa} described all representations  of $B_{\Theta}^{0,n}$ for $n=2$. However, for $n>2$, its representation theory is not known.  We first  describe a general form of  a  representation, say $\pi$, of $B_{\Theta}^{m,n}$ in terms of certain parameters. Using this and a truncation technique, we prove that any homomorphic image of the ideal $J$ generated by the defect projection of the isometry $s_1s_2\cdots s_{m+n}$ is stable.  Exploiting the universal properties of $B_{\Theta}^{m,n}$ and irrational noncommtative torus, we produce a  short exact sequence of $C^*$-algebras whose middle $C^*$-algebra is  $\pi(B_{\Theta}^{m,n})$, end $C^*$-algebra is noncommutative torus, and the associated kernel is $J$.     Invoking the Five lemma, we prove $K$-stability of any homomorphic image of $B_{\Theta}^{m,n}$. As a consequence,  we get $K$-stability of $B_{\Theta}^{m,n}$ as well as  its closed ideals.

In  section $5$, we take  a $n \choose 2$-tuple $\mathcal{U}=\{U_{ij}\}_{1\leq  i<j\leq n}$ of  commuting unitaries  with joint spectrum $X \subset \bbbt^{n \choose 2}$.
Consider  the $C^*$-algebra $A$ generated by an $n$-tuple of $\mathcal{U}$-twisted isometries.
We first establish $A$ as a continuous $C(X)$-algebra with fiber isomorphic to a homomorphic image of $C_{\Theta}^{0,n}$.  Under the assumption that $X$ does not contain any degenerate skew-symmetric matrix, we show that $A$ is $K$-stable.  Finally, we prove the $K$-stability  for the
$C^*$-algebra generated by an $n$-tuple of free $\mathcal{U}$-twisted isometries. In the last section, we discuss some further directions for investigation.

Throughout the paper, all algebras and Hilbert spaces are assumed to be separable and defined over the field $\mathbb{C}$.  Let $\mathbb{N}_0=\mathbb{N} \cup \{0\}$.
The set  $\{e_{n}:n\in \mathbb{N}_0\}$ denotes the standard orthonormal basis for the Hilbert space $\ell^{2}(\mathbb{N}_0)$. The letter $p_{ij}$ denotes the rank one operator mapping $e_i$ to $e_j$. We denote $p_{00}$ by $p$. The letter $S$ denotes the unilateral shift $e_{n}\mapsto e_{n-1}$ on $\ell^{2}(\mathbb{N}_0)$. The number operator $e_{n}\mapsto ne_{n}$ on $\ell^{2}(\mathbb{N}_0)$ is denoted by $N$. The closed ideal of an algebra generated by elements $a_{1},\cdots,a_{n}$ is denoted by $\langle a_{1},\cdots,a_{n}\rangle$. The Toeplitz algebra generated by the unilateral shift is denoted by $\mathcal{T}$. The symbol $\mathcal{K}$ is reserved for the algebra of compact operators.
By the symbol $\overrightarrow{\prod_{j=1}^{n}}s_{i}$, we mean $s_{1}\cdots s_{n}$.
The parameter  $\Theta=\{\theta_{ij} \in \mathbb{R}:1\leq i<j\leq n\}$ denote a $n \choose 2$-tuple of real numbers. Let $M_{\Theta}$ be the associated $n\times n$ skew-symmetric matrix such that the $ij$-th entry of  $M_{\Theta}$ is $\theta_{ij}$ for $i<j$. Let $\bigwedge_n$ be the set of $n \times n$ nondegenerate  skew-symmetric matrices (see \cite{Phi-2006aa}).
For any subset $I$ of  $\{1,2,\cdots ,n\}$, define
$\Theta_I=\{\theta_{ij}: i<j, i,j \in I\}$. For $1\leq m \leq n$,
$$\Theta_{[m]}=\{\theta_{ij}:   1 \leq i< j \leq m\}\,  \mbox{ and } \Theta_{[\hat{m}]}=\{\theta_{ij}: i<j,\,   i \neq m, j \neq m\}.$$

A word of caution: the number $n$
may vary in different cases, and consequently, the sizes of the associated matrices and tuples may also change. However, the value of $n$ will always be clear from the context.
\newsection{Tensor twist of  isometries}
In this section, we give a full description of irreducible representations of the $C^*$-algebras $C_{\Theta}^{m,n}$. All these results can be derived from (\cite{MarPau-2016aa}, \cite{Mor-2013aa},  \cite{NarJaySur-2022aa}).  However, to make the paper self-contained, we provide a brief sketch of its proof.

\begin{definition} \label{isometry}
	Set $C^{1,0}=C(\bbbt)$ and $C^{0,1}=\mathcal{T}$. For $m+n>1$ and  $\Theta=\{\theta_{ij}:1 \leq i <j \leq m+n\}$,  define $C_{\Theta}^{m,n}$ to be the universal $C^*$-algebra generated by
	$s_1, s_2, \cdots s_{m+n}$ satisfying the following relations;
	\begin{IEEEeqnarray*}{rCll}
		s_i^*s_j&=&e^{-2\pi \mathrm{i} \theta_{ij} }s_js_i^*,\,\,& \mbox{ if } 1\leq i<j \leq m+n;\\
		s_i^*s_i&=& 1,\,\, & \mbox{ if } 1\leq i\leq m+n;\\
		s_is_i^*&=&1, \,\, & \mbox{ if }  1\leq i \leq m.
	\end{IEEEeqnarray*}
\end{definition}

\begin{remark} \label{remark1}
	Note that
	\begin{enumerate}[(i)]
		\item The $C^*$-algebra $C_{\Theta}^{m+n,0}$ is the noncommtative $(m+n)$-torus. We write $\mathcal{A}_{\Theta}^{m+n}$ for $C_{\Theta}^{m+n,0}$.
		\item Following \cite{Mor-2013aa}, we call the $C^*$-algebra $C_{\Theta}^{0,n}$ the universal $C^*$-algebra generated by a tensor twist of $n$ isometries.
		We denote $C_{\Theta}^{0,n}$ by $C_{\Theta}^{n}$.
		
		\item There is a chain of  canonical maps $\beta_l$, $1\leq l \
		\leq m+n$;
		\[
		C_{\Theta}^{0,m+n} \xrightarrow{ \beta_0 } C_{\Theta}^{1,m+n-1} \xrightarrow{\beta_1} C_{\Theta}^{2,m+n-2} \cdots \cdots
		\xrightarrow{\beta_{m+n-1}} C_{\Theta}^{m+n,0}.
		\]
		mapping the canonical generators of $C_{\Theta}^{l-1,m+n-l+1}$ to the canonical generators of $C_{\Theta}^{l,m+n-l}$.
		\item The ordering of the generator and the paramater $\Theta$ are related as follows.
		
		If we change the order of the generators by a permutation $P$ then the associated skew symmetric matrix will be $PM_{\Theta}P$, whose strictly upper triangular entries will be   the   replacement for the parameter $\Theta$.
	\end{enumerate}
\end{remark}
\begin{remark} We will denote the standard generators $s_i$ of $C_{\Theta}^{m,n}$ by $s_i^{m,n}$ whenever there is a possibility of confusion.
\end{remark}

\noindent  We will now describe all irreducible representations of $C_{\Theta}^{m,n}$. Denote by  $\mathcal{A}^{m}_{\Theta}$  the $m$-dimensional rotation algebra  with parameter $\Theta$ if $m >1$ and $C(\mathbb{T})$ if $m=1$. 	Fix $m,n \in \mathbb{N}_0$ such that $m+n\geq 1$,
define   $$\Sigma_{m,n}=\{I\subset \{1,2,\cdots ,m+n\}: \{1,2,\cdots m\} \subset I\} \mbox{ and } \Theta_I=\{\theta_{ij}:1 \leq i<j \leq m+n, i,j\in I\}.$$
Fix $I=\{i_1<i_2<\cdots  < i_r\} \in \Sigma_{m+n}$. Let $I^c=\{j_1<j_2<\cdots < j_s\}$. 	Let $\rho:\mathcal{A}_{\Theta_I}^{|I|} \rightarrow \mathcal{L}(K)$ be a unital representation. Take
\[
\mathcal{H}^{I}= K
\otimes \underbrace{\ell^2(\mathbb{N}_0)\otimes \ell^2(\mathbb{N}_0)\otimes \cdots \otimes\ell^2(\mathbb{N}_0)}_{m+n-|I| \mbox{ copies }}.
\]
Define a map $\pi_{(I,\rho)}$ of $C_{\Theta}^{m,n}$ as follows:
\begin{IEEEeqnarray*}{rCll}
	\pi_{(I,\rho)} :C_{\Theta}^{m,n} &\rightarrow &\mathcal{L}(\mathcal{H}^{I})\\
	s_{j_l} &\mapsto&
	1 \otimes 1^{\otimes^{l-1}}\otimes S^* \otimes e^{2\pi \textrm{i} \theta_{j_lj_{l+1}} N}\otimes  \cdots \otimes e^{2\pi \textrm{i} \theta_{j_lj_{s}} N}, &\quad \mbox{ for } 1 \leq l \leq s, \\
	s_{i_l} &\mapsto&
	\pi(s_{i_l}) \otimes \lambda_{i_l,j_1}\otimes  \lambda_{i_l,j_2} \otimes  \cdots \otimes  \lambda_{i_l,j_s}&\quad \mbox{ for } 1 \leq l \leq r,
\end{IEEEeqnarray*}
where $\lambda_{i_l,j_k}$ is  $e^{-2\pi \textrm{i} \theta_{i_l,j_k} N} $ if $i_l > j_k$ and $e^{2\pi \textrm{i} \theta_{i_l,j_k} N}$  if $i_l < j_k$. Since  $	\{\pi_{(I,\rho)}(s_i)\}_{1\leq i \leq m+n}$ satisfy the defining relations of $C_{\Theta}^{m,n}$, it follows that 	$\pi_{(I,\rho)}$  is a representation of $C_{\Theta}^{m,n}$. Also,  if $\rho$ is irreducible then  $\pi_{(I,\rho)}$ is also  irreducible  as by taking action of  appropriate operators and using irreducibility of $\rho$, one can show that any invariant subspace of $\mathcal{H}^I$ contains  the subspace spanned by $\{h \otimes e_0\otimes \cdots e_0: h\in K\}$.  Moreover, if $\rho$ and $\rho^{\prime}$ are unitarily equivalent then so are  $\pi_{(I,\rho)}$ and  $\pi_{(I,\rho^{\prime})}$.
In what follows, we will give a sketch of the proof that these  irreducible representations exhaust the set of all irreducible representations of $C_{\Theta}^{m,n}$ up to unitary equivalence. We refer the reader to \cite{NarJaySur-2022aa} for more details. 	

\begin{theorem}   \label{representations}
	The set $\{\pi_{(I,\rho)}: I \subset \Sigma_{m,n}, \rho \in \widehat{\mathcal{A}_{\Theta_I}^{|I|}}\}$ gives all irreducible representations of $C_{\Theta}^{m,n}$ upto unitarily equivalence.
\end{theorem}
\begin{proof}
	It suffices to show that any irreducible representation  of $C_{\Theta}^{m,n}$ is unitarily equivalent to  $\pi_{(I,\rho)}$ for some $I \subset \Sigma_{m,n}$ and $ \rho \in \widehat{\mathcal{A}_{\Theta_I}^{|I|}}$. For that, take  $\pi$  to be an irreducible  representation of $C_{\Theta}^{m,n}$ acting on the Hilbert space $\mathcal{H}$.  Let $T_i=\pi(s_i)$ for $1 \leq i \leq m+n$.  Then it follows from Theorem $3.6$ \cite{NarJaySur-2022aa} that the tuple $(T_1,T_2, \cdots T_{m+n})$ admits the von-Neumann Wold decomposition. Therefore,  we have
	\begin{enumerate}[(i)]
		\item $\mathcal{H}=\oplus_{I \subset \Sigma_{m,n}}\mathcal{H}_I,$
		\item $	\restr{T_j}{\mathcal{H}_I}$  is a shift if $ j \notin I$, and $	\restr{T_j}{\mathcal{H}_I}$  is a unitary if $ j \in I$,
		\item $\mathcal{H}_I$ are reducing subspace of $\pi$ for $I \subset \Sigma_{m,n}$,
		\item $\mathcal{H}_I= \oplus_{\ell \in \mathbb{N}_0^{m+n-|I|}} \textbf{T}_{I^{c}}^{\ell}V_{I^c}$,
		where $V_{I^c}=\cap_{\ell \in \mathbb{N}_0^{|I|}}\,\textbf{T}_{I^{c}}^{\ell}(\cap_{j\notin I} \ker T_j^*)$ (see Theorem $3.6$ \cite{NarJaySur-2022aa}).
	\end{enumerate}
	Since $\pi$ is irreducible, there exists $I$ such that $\mathcal{H}_I $ is nontrivial, and for $I^{\prime} \neq I$, one has $\mathcal{H}_{I^{\prime}}=\{0\}$.  Moreover, using the commutation relations, it follows that  $V_{I^c}$ is an invariant subspace for $\{s_i:i\in I\}$.
	Let $C_{I}$ be the $C^*$-algebra generated by $\{T_i:i\in I\}$.  It is not difficult to verify that the generators $\{T_i\}_{i \in I}$   satisfy the defining   commutation relations of  $\mathcal{A}_{\Theta_{I}}^{|I|}$ and hence there exists a surjective homomorphism  from $\mathcal{A}_{\Theta_{I}}^{|I|}$ to $C_I$.  Using this, one can  define
	$$\rho:\mathcal{A}_{\Theta_{I}}^{|I|}\rightarrow C_I \rightarrow \mathcal{L}(V_{I^c}); \quad  s_i \mapsto T_i \mapsto  \restr{T_i}{V_{I^c}}=\restr{\pi(s_i)}{V_{I^c}}, \quad \mbox{ for }  i \in I.$$
	Then $\rho$ is an irreducible  representation of $\mathcal{A}_{\Theta_{I}}^{|I|}$. Moreover, from part $(iv)$, one can see that
	$$ \mathcal{H}_I\cong V_{I^c} \otimes \underbrace{\ell^2(\mathbb{N}_0)\otimes \ell^2(\mathbb{N}_0)\otimes \cdots \otimes\ell^2(\mathbb{N}_0)}_{m+n-|I| \mbox{ copies }}.$$
	Using these facts,  it is straightforward to see that $\pi$ is unitarily equivalent to $\pi_{I,\rho}$.
\end{proof}

\begin{proposition} \label{injective} Let $I \in \Sigma_{m,n}$.  Then the  map $\Phi_I: \mathcal{A}_{\Theta_{I}}^{|I|} \rightarrow C_{\Theta}^{|I|,m+n-|I|}$ sending $s_i^{m,0} \mapsto s_i^{|I|,m+n-|I|}$ is an injective homomorphism.
\end{proposition}
\begin{proof} Without loss of generality, we will assume that  $I=\{1,2,\cdots |I|\}$. Let     $C_I$ be the $C^*$-subalgebra of $C_{\Theta}^{m,n}$ generated by $s_i^{|I|,m+n-|I|}; i \in I$. By restricting the codomain, we get the following homomorphism
	$$ \Phi_I: \mathcal{A}_{\Theta_{I}}^{|I|} \rightarrow C_I; \quad s_i^{|I|,0} \mapsto s_i^{|I|,m+n-|I|}.$$
	To prove the claim, it is enough to show that any representation $\rho$ acting on $K$  factors through $\Phi_I$. Observe that,
	$$ \pi_{(I,\rho)}(s^{|I|,m+n-|I|}_i)(h \otimes e_0\otimes  \cdots \otimes e_0)= \rho(s_i^{|I|,0})(h) \otimes e_0 \cdots \otimes e_0; \mbox { for } 1\leq i \leq |I|.$$
	Thus,  by identifying $K$ with $K \otimes e_0\otimes \cdots \otimes e_0$, we get the following commutative diagram.
	\[
	\begin{tikzcd}
		\mathcal{A}_{\Theta_{I}}^{|I|} \arrow[rd, "\rho" '] \arrow[r, "\Phi_I"] & C_I \arrow[d, "\pi_{(I,\rho)}"] \\
		& \mathcal{L}(K)
	\end{tikzcd}
	\]
	This proves the claim.\end{proof}
\begin{proposition} \label{exact-sequence1}
	Let  $m,n \in \mathbb{N}_0$ such that  $m+n \geq 1$ and let  $1 \leq l \leq n$.  Let $J_l^{m,n}$ denote the ideal of $C_{\Theta}^{m,n}$ generated by the defect projection of  $ \prod_{1 \leq i \leq l}s_{m+i}^{m,n}$. Then one has the following short exact sequence $\chi_{m,n}^l$ of $C^*$-algebras.
	\[
	\chi_{m,n}^l: \quad 0\longrightarrow J_l^{m,n} \xrightarrow{i} C_{\Theta}^{m,n}\xrightarrow{\beta_{m}^l}  C_{\Theta}^{m+l,n-l}\longrightarrow  0,
	\]
	where $\beta_m^l=\beta_{m+l-1}\circ \cdots \circ\beta_{m}$.
\end{proposition}
\begin{proof}
	It is enough to show that $\ker (\beta_{m}^l)=J_l^{m,n}$. Clearly, $J_l^{m,n} \subset \ker (\beta_{m}^l)$. Now, consider
	a representation $\zeta$ of  $C_{\Theta}^{m,n}$ on the Hilbert space $\mathcal{H}$ which vanishes on $J_l^{m,n}$.
	This implies that  $\{\zeta(s_i):1\leq \textrm{i} \leq m+n\}$ satisfy the defining relations of $C_{\Theta}^{m+l,n-l}$.
	By the universal property of  $C_{\Theta}^{m+l,n-l}$,
	we get a representation $\varsigma$ of $C_{\Theta}^{m+l,n-l}$ such that $\varsigma\circ \beta_m^l=\zeta$. This proves that $\ker (\beta_{m}^l) \subset J_l^{m,n}$ and hence the claim.
\end{proof}

\begin{theorem} \label{faithful representation}
	Let $ \rho: \mathcal{A}_{\Theta_{[m]}}^{m} \rightarrow \mathcal{L}(\mathcal{H})$ be a faithful representation of $\mathcal{A}_{\Theta_{[m]}}^{m}$.  Then $\pi_{I,\rho}$ is a faithful representation of $C_{\Theta}^{m,n}$.
\end{theorem}
\begin{proof} It is not difficult to see  that  if a representation $\tilde{\rho}$ factors through $\rho$ then   $\pi_{I,\tilde{\rho}}$ factors through $\pi_{I,\rho}$. Hence by Theorem (\ref{representations}),  all irreducible representations of $C_{\Theta}^{m,n}$ factors through $\pi_{I,\rho}$, proving the claim.
\end{proof}	

\begin{corollary} \label{kernel1} Let $ \rho: \mathcal{A}_{\Theta_I}^{|I|} \rightarrow \mathcal{L}(\mathcal{H})$ be a faithful representation of $\mathcal{A}_{\Theta_I}^{|I|}$. Then $\ker \pi_{I,\rho}=J_I^{m,n}$.
\end{corollary}
\begin{proof} It is an immediate consequence of Theorem (\ref{faithful representation}) and Proposition (\ref{exact-sequence1}).
\end{proof}	
Define a topology $\mathscr{T}$ on $\Sigma_{m,n}$ as follows. Call a subset $Z$ open if whenever $I\subset Z$ then $I^{\prime} \subset Z$ for every subset $I^{\prime} \subset I$.
\begin{corollary} \label{topology}
	Let $m,n\in \mathbb{N}_0$ such that $m \geq 2$. Let $\Theta \in \bigwedge_{m+n}$ such that $\Theta_I \in\bigwedge_{|I|}$ for all $I \in \Sigma_{m,n}$. For $I \in \Sigma_{m,n}$, let $J_I$ be the ideal of $C_{\Theta}^{m,n}$  generated by the defect projection of the isometry $\prod_{i \in I} s_i$. Then
	\[
	\mbox{Prim}(C_{\Theta}^{m,n})=\{J_I: I \in \Sigma_{m,n}\}.
	\]
	Moreover, the hull-kernel topology on 	$\mbox{Prim}(C_{\Theta}^{m,n})$ is same as $\tau$.
\end{corollary}
\begin{proof} Note that,  $\mathcal{A}_{\Theta_I}^{|I|}$ is simple for all  $I \in \Sigma_{m,n}$. Hence by Corollary \ref{kernel1}, we get the claim.
\end{proof}	

Define the Hilbert space
\[
\mathcal{H}^{m,n}=\underbrace{\ell^2(\mathbb{Z})\otimes \ell^2(\mathbb{Z})\otimes \cdots \otimes\ell^2(\mathbb{Z})}_{m \mbox{ copies }}
\otimes \underbrace{\ell^2(\mathbb{N})\otimes \ell^2(\mathbb{N})\otimes \cdots \otimes\ell^2(\mathbb{N})}_{n \mbox{ copies }}.
\]

For $1 \leq i \leq m+n$, define the operator $s_i^{m,n}$ to be an operator on $\mathcal{H}^{m,n}$ given as follows:
\[
T_i^{m,n}= 1^{\otimes^{i-1}}\otimes S^* \otimes
e^{2\pi \textrm{i} \theta_{i,i+1}N}\otimes e^{2\pi \textrm{i} \theta_{i,i+2}N}\otimes \cdots \otimes e^{2\pi \textrm{i} \theta_{i,m+n}N}.
\]
Since $\{T_i^{m,n}\}_{1\leq i \leq m+n}$ satisfy the defining relations of $C_{\Theta}^{m,n}$, we get a representation
\[
\Psi^{m,n}:C_{\Theta}^{m,n} \rightarrow \mathcal{L}(\mathcal{H}^{m,n}); \quad  s_i^{m,n} \mapsto T_i^{m,n}.
\]

\begin{theorem} \label{faithful}
	The representation $\Psi^{m,n}$ is faithful.
\end{theorem}
\begin{proof} One can see that $\Psi^{m,n}=\pi_{I,\rho} $, where $I=\{1,2,\cdots m\}$ and
	$$\rho:  \mathcal{A}_{\Theta_{[m]}}^{m} \longrightarrow \mathcal{L}(\ell^2(\mathbb{Z})^{\otimes m});$$
	given by
	\[ \rho(s_i^{m,n})= 1^{\otimes^{i-1}}\otimes S^* \otimes
	e^{2\pi \textrm{i} \theta_{i,i+1}N}\otimes e^{2\pi \textrm{i} \theta_{i,i+2}N}\otimes \cdots \otimes e^{2\pi \textrm{i} \theta_{i,m}N}, \mbox{ for } 1 \leq i \leq m.
	\]
	The claim follows by Theorem (\ref{faithful representation}) and the fact that $\rho$ is faithful.
\end{proof}

From now on, we will identify $C_{\Theta}^{m,n}$ with its image under the map $ \Psi^{m,n}$ and the generators $s_i^{m,n}$ of $C_{\Theta}^{m,n}$
with $T_i^{m,n}$ for each
$1\leq i \leq m+n$.

\section{Stable and Non-stable $K$-groups of $C_{\Theta}^{m,n}$}
In this section, we first compute $K$ groups of $C_{\Theta}^{m,n}$ for all $m,n \in \mathbb{N}_{0}$. Then we recall from (\cite{SetVai-2020aa},\cite{Tho-1991aa}) the notion of $K$-stability, and  prove that the $C^*$-algebra  $C_{\Theta}^{m,n}$ is $K$-stable if $m+n>1$ and
$\Theta  \in \bigwedge_{m+n}$.

\begin{proposition} \label{exact-sequence}
	Let  $m,n \in \mathbb{N}_0$ such that  $m+n \geq 1$ and let  $1 \leq l \leq n$.  Let $I_{m,n}$ denote the ideal of $C_{\Theta}^{m,n}$ generated by the defect projection of  $s_{m+1}^{m,n}$. Then one has the following short exact sequence $\chi_{m,n}$ of $C^*$-algebras.
	\begin{IEEEeqnarray}{rCl}\label{exactsequence}
		\chi_{m,n}: \quad 0\longrightarrow I_{m,n} \xrightarrow{i} C_{\Theta}^{m,n}\xrightarrow{\beta_m}  C_{\Theta}^{m+1,n-1}\longrightarrow  0.
	\end{IEEEeqnarray}
\end{proposition}
\begin{proof} The proof follows  from Proposition (\ref{exact-sequence1}) by putting $l=1$.
\end{proof}	

\begin{proposition} \label{twisted}
	Let $I_{m,n}$ be the ideal of $C_{\Theta}^{m,n}$ generated by $1-s_{m+1}^{m,n}(s_{m+1}^{m,n})^*$. Then one has
	\[I_{m,n}
	\cong \mathcal{K}\otimes C_{\Theta_{[\widehat{m+1}]}}^{m,n-1}.
	\]
\end{proposition}
\begin{proof} Define an automorphism $\varphi: C_{\Theta_{[\widehat{m+1}]}}^{m,n-1} \rightarrow C_{\Theta_{[\widehat{m+1}]}}^{m,n-1}$ given by;
	\[
	\varphi(s_i^{m,n-1})= \big(\Pi_{l=m+2}^{m+n-1}e^{2\pi \textrm{i} \theta_{m+1,l}}\big) s_i^{m,n-1}
	\]
	for $1\leq i \leq m+n, i \neq m+1$.  Using the definition $(3.1)$ given in \cite{Mor-2013aa}, one can verify that
	\[
	I_{m,n}\cong \mathcal{K}\otimes_{\varphi} C_{\Theta_{[\widehat{m+1}]}}^{m,n-1}.
	\]
	By applying Lemma $2.2$ in \cite{Mor-2013aa}, we get
	\[
	I_{m,n}\cong \mathcal{K}\otimes_{\varphi} C_{\Theta_{[\widehat{m+1}]}}^{m,n-1} \cong \mathcal{K}\otimes C_{\Theta_{[\widehat{m+1}]}}^{m,n-1}.
	\]
\end{proof}

\begin{remark} \label{rem} Note that by Lemma $2.2$ in \cite{Mor-2013aa},  the element $1-s_{m+1}^{m,n}(s_{m+1}^{m,n})^* \in C_{\Theta}^{m,n}$ can be identified with $p \otimes 1 \in \mathcal{K} \otimes C_{\Theta_{[\widehat{m+1}]}}^{m,n-1}$ under the above isomorohism.
\end{remark}
Consider the map $\tau_{m,n}: C_{\Theta_{[\widehat{m+1}]}}^{m,n} \longrightarrow C_{\Theta_{[\widehat{m+1}]}}^{m,n}$
given by
\[
\tau_{m,n}(s_i^{m,n})=\begin{cases}
	s_{m+1}^{m+1,n}\Psi^{m+1,n}(s_i^{m+1,n})(s_{m+1}^{m+1,n})^* & \mbox{ if } 1 \leq i \leq m; \cr
	s_{m+1}^{m+1,n}\Psi^{m+1,n}(s_{i+1}^{m+1,n})(s_{m+1}^{m+1,n})^* & \mbox{ if } m+1 \leq i \leq m+n; \cr
\end{cases}
\]
Using the explicit description of the operators $T_i^{m,n}$ in the previous section, it is not difficult to verify that  $\{\tau_{m,n}(s_i^{m,n}): 1\leq i \leq m+n\}$ satisfy the defining relations (\ref{isometry}) of $C_{\Theta_{[\widehat{m+1}]}}^{m,n}$.
This shows  that $\tau_{m,n}$ is an automorphism of $C_{\Theta_{[\widehat{m+1}]}}^{m,n}$.

\begin{proposition}\label{cross-product}
	Let $m,n \in \mathbb{N}_0$ such that $m+n\geq 1$. Let $\Theta \in \mathbb{R}^{m+n \choose 2}$. Then
	one has the following.
	\[
	C_{\Theta_{[\widehat{m+1}]}}^{m,n}\rtimes_{\tau_{m,n}} \mathbb{Z} \cong  C_{\Theta}^{m+1,n}.
	\]
\end{proposition}

\begin{proof} By \cite[Corollary 9.4.23]{NCPhil}, it follows that the $C^{*}$-algebra $C_{\Theta_{[\widehat{m+1}]}}^{m,n}\rtimes_{\tau_{m,n}} \mathbb{Z}$ is the universal $C^{*}$ algebra generated by $C_{\Theta_{[\widehat{m+1}]}}^{m,n}$ and a unitary $u$ subject to the relation $uau^{*}=\tau_{m,n}(a)$ for all $a \in C_{\Theta_{[\widehat{m+1}]}}^{m,n}$.
	From the definition of the map $\tau_{m,n}$, one can verify that
	\[
	us_i^{m,n}u^*=\begin{cases}
		e^{-2\pi \textrm{i}\theta_{i,m+1}N}s_i^{m,n} & \mbox{ if }  i \leq m; \cr
		e^{2\pi \textrm{i}\theta_{m+1,i+1}N}s_i^{m,n} & \mbox{ if } i>m; \cr
	\end{cases}
	\]
	This gives the following relations.
	\[
	u^*s_i^{m,n}=\begin{cases}
		e^{2\pi \textrm{i}\theta_{i,m+1}N}s_i^{m,n}u^* & \mbox{ if }  i \leq m; \cr
		e^{-2\pi \textrm{i}\theta_{m+1,i+1}N}s_i^{m,n}u^*& \mbox{ if } i>m; \cr
	\end{cases}
	\]
	Therefore we get  an isomorphism $\phi:C_{\Theta_{[\widehat{m+1}]}}^{m,n}\rtimes_{\tau_{m,n}} \mathbb{Z}\rightarrow C_{\Theta}^{m+1,n}$  mapping
	$$ u \mapsto 	s_{m+1}^{m+1,n} \qquad \mbox{ and }  \qquad s_{i}^{m,n}\rightarrow \begin{cases}
		s_{i}^{m+1,n}, & \mbox{ if } 1 \leq i \leq m, \cr
		s_{i+1}^{m+1,n}, & \mbox{ if } m+1 \leq i \leq m+n. \cr
	\end{cases}$$
\end{proof}

\begin{remark} Note that these results implicitly involve a reindexing of parameters.
\end{remark}
\begin{theorem}\label{$K$-groups}
	Let $n \in \mathbb{N}$. Then we have
	$$K_0\big(C_{\Theta}^{0,n})\big)= \langle [1]\rangle \cong \mathbb{Z} \, \mbox{ and }\, K_1\big(C_{\Theta}^{0,n})\big)=0.$$
\end{theorem}

\begin{proof} First note that $C^{0,1}=\mathcal{T}$.  Hence  $K_0\big(C^{0,1})\big)=\mathbb{Z}$ generated by $[1]$ and $K_1\big(C^{0,1})\big)=\{0\}$ (see \cite{RorLarLau-2000aa}). Assume  that $K_0\big(C_{\Theta_{[\widehat{1}]}}^{0,n-1})\big)=\mathbb{Z}$ generated by $[1]$ and $K_1\big(C_{\Theta_{[\widehat{1}]}}^{0,n-1})\big)=\{0\}$. From Proposition (\ref{cross-product}), we have
	\[
	C_{\Theta_{[\widehat{1}]}}^{0,n-1}\rtimes_{\tau_{0,n-1}} \mathbb{Z} \cong  C_{\Theta}^{1,n-1}.
	\]
	The associated Pimsner-Voiculescu six term exact sequence  of $K$-groups is given by the following commutative diagram (see \cite{Bla-1998aa} for details).
	\begin{center}
		\begin{tikzpicture}[node distance=1cm,auto]
			\tikzset{myptr/.style={decoration={markings,mark=at position 1 with %
						{\arrow[scale=2,>=stealth]{>}}},postaction={decorate}}}
			\node (A){$ K_0(C_{\Theta_{[\widehat{1}]}}^{0,n-1})$};
			\node (Up)[node distance=2cm, right of=A][label=above:$id-(\tau_{0,n-1}){*}$]{};
			\node (B)[node distance=5cm, right of=A]{$K_0(C_{\Theta_{[\widehat{1}]}}^{0,n-1})$};
			\node (Up)[node distance=2cm, right of=B][label=above:$i_*$]{};
			\node (C)[node distance=5cm, right of=B]{$K_0(C_{\Theta_{[\widehat{1}]}}^{0,n-1}\rtimes_{\tau_{0,n-1}} \mathbb{Z})$};
			\node (D)[node distance=2cm, below of=C]{$K_1(C_{\Theta_{[\widehat{1}]}}^{0,n-1})$};
			\node (E)[node distance=5cm, left of=D]{$K_1(C_{\Theta_{[\widehat{1}]}}^{0,n-1})$};
			\node (F)[node distance=5cm, left of=E]{$K_1(C_{\Theta_{[\widehat{1}]}}^{0,n-1}\rtimes_{\tau_{0,n-1}} \mathbb{Z})$};
			\node (Below)[node distance=2cm, left of=D][label=below:$id-(\tau_{0,n-1}){*}$]{};
			\node (Below)[node distance=2cm, left of=E][label=below:$i_*$]{};
			\draw[myptr](A) to (B);
			\draw[myptr](B) to (C);
			\draw[myptr](C) to node{{ $\delta$}}(D);
			\draw[myptr](D) to (E);
			\draw[myptr](E) to (F);
			\draw[myptr](F) to node{{ $\partial$}}(A);
		\end{tikzpicture}
	\end{center}
	\noindent
	Since  $\tau_{0,n-1} \sim_{h} 1$, we have  $id-(\tau_{0,n-1}){*}=0$. Using this and the induction hypothesis, it follows that  the maps $i_{*}$ and $\partial$ are  isomorphisms. Hence we have
	$$K_0\big(C_{\Theta}^{1,n-1})\big)=
	\langle [1] \rangle \cong \mathbb{Z}\quad  \mbox { and } \quad  K_1\big(C_{\Theta}^{1,n-1})\big)=\langle [s_1^{1,n-1}] \rangle \cong \mathbb{Z}.$$
	The  six term exact sequence of $K$-groups associated with the exact sequence  of Proposition (\ref{exact-sequence}) is given by the following  diagram.
	\begin{center}
		\begin{tikzpicture}[node distance=1cm,auto]
			\tikzset{myptr/.style={decoration={markings,mark=at position 1 with %
						{\arrow[scale=2,>=stealth]{>}}},postaction={decorate}}}
			\node (A){$ K_0(I_n)$};
			\node (Up)[node distance=2cm, right of=A][label=above:$i_*$]{};
			\node (B)[node distance=4cm, right of=A]{$K_0(C_{\Theta}^{0,n})$};
			\node (Up)[node distance=2cm, right of=B][label=above:$(\beta_1)_*$]{};
			\node (C)[node distance=5cm, right of=B]{$K_0(C_{\Theta}^{1,n-1})$};
			\node (D)[node distance=2cm, below of=C]{$K_1(I_n)$};
			\node (E)[node distance=4cm, left of=D]{$K_1(C_{\Theta}^{0,n})$};
			\node (F)[node distance=5cm, left of=E]{$K_1(C_{\Theta}^{1,n-1})$};
			\node (Below)[node distance=2cm, left of=D][label=below:$i_{*}$]{};
			\node (Below)[node distance=2cm, left of=E][label=below:$(\beta_1)_*$]{};
			\draw[myptr](A) to (B);
			\draw[myptr](B) to (C);
			\draw[myptr](C) to node{{ $\delta$}}(D);
			\draw[myptr](D) to (E);
			\draw[myptr](E) to (F);
			\draw[myptr](F) to node{{ $\partial$}}(A);
		\end{tikzpicture}
	\end{center}
	From the induction hypothesis, Proposition (\ref{twisted}) and Remark \ref{rem}, it follows that  $K_0\big(I_n\big)=\mathbb{Z}$ generated by $[p\otimes 1]=[1-s_{1}^{0,n}(s_{1}^{0,n})^{*}]$ and
	$K_1\big(I_n\big)=0$. Moreover, we have
	\[
	\partial([s_1^{1,n-1}] )=[1-s_{1}^{0,n}(s_{1}^{0,n})^{*}], \quad (\beta_1)_*([1])=[1]  \quad \mbox{ and } \quad \delta \equiv 0
	\]
	Using these facts, we get the claim.
\end{proof}

By \cite{Rie-1987aa}, it follows that  both the $K$-groups of  $\mathcal{A}_{\Theta_{[m]}}^{m}$ are isomorphic to $ \mathbb{Z}^{2^{m-1}}$. We denote the generators of $K_0\big(\mathcal{A}_{\Theta_{[m]}}^{m}\big)$ by $\{[\mathbb{P}_i]: 1 \leq i \leq 2^{m-1}\}$ and the generators of $K_1\big(\mathcal{A}_{\Theta_{[m]}}^{m}\big)$ by $\{[\mathbb{U}_i]: 1 \leq i \leq 2^{m-1}\}$. If we need to specify the generators $v_{1},...,v_{m}$ of $\mathcal{A}_{\Theta_{[m]}}^{m}$, then we denote the generators of $K_0\big(\mathcal{A}_{\Theta_{[m]}}^{m})\big)$ by $\{[\mathbb{P}_i(v_{1},\cdots,v_{m})]: 1 \leq i \leq 2^{m-1}\}$ and the generators of $K_1\big(\mathcal{A}_{\Theta_{[m]}}^{m})\big)$ by $\{[\mathbb{U}_i(v_{1},\cdots,v_{m})]: 1 \leq i \leq 2^{m-1}\}$. In the following, we treat
$C_{\Theta_{[m]}}^{m,0}$  as a subalgebra of $C_{\Theta}^{m,n}$ generated by $s_1^{m,n},s_2^{m,n}, \cdots s_m^{m,n}$, and hence $[\mathbb{P}_i]$ and $[\mathbb{U}_i]$ are elements of $K_0\big(C_{\Theta}^{m,n})\big)$  and $K_1\big(C_{\Theta}^{m,n})\big)$, respectively. However, one needs to check whether these classes are nontrivial, and the following theorem establishes this.

\begin{theorem} \label{K-groups general}
	Let  $m,n \in \mathbb{N}_0$ such that $m+n> 1$. Let $\Theta \in \mathbb{R}^{m+n \choose 2}$.
	Then we have
	$$K_0\big(C_{\Theta}^{m,n})\big)=
	\begin{cases}
		\mathbb{Z}^{2^{m-1}} & \mbox{ if } m\geq 1, \cr
		\mathbb{Z}& \mbox{ if } m =0. \cr
	\end{cases}
	\quad \mbox{ and } \quad
	K_1\big(C_{\Theta}^{m,n})\big)=
	\begin{cases}
		\mathbb{Z}^{2^{m-1}} & \mbox{ if } m\geq 1, \cr
		0& \mbox{ if } m =0. \cr
	\end{cases}.$$
	Moreover,  $K_0\big(C_{\Theta}^{m,n})\big)$ is generated by   $\{[\mathbb{P}_i]: 1 \leq i \leq 2^{m-1}\}$ and $K_1\big(C_{\Theta}^{m,n})\big)$ is generated by  $\{[\mathbb{U}_i]: 1 \leq i \leq 2^{m-1}\}$.
\end{theorem}
\begin{proof} Let $\Theta_{> m}=\{\theta_{ij}: m+1 \leq i<j\leq m+n\}$. Using  Proposition (\ref{cross-product}) iteratively, we get
	\[
	C_{\Theta}^{m,n}
	\cong \Big(\big((C_{\Theta_{> m}}^{0,n}\rtimes_{\tau_{0,n}} \mathbb{Z})\rtimes_{\tau_{1,n}}\mathbb{Z} \big) \cdots  \rtimes_{\tau_{m-1,n}} \mathbb{Z}\Big).
	\]
	The claim follows by Theorem \ref{$K$-groups} and a similar calculations as done in case of non-commutative torus (\cite{PimVoi-1980aa}) using Pimsner-Voiculescu six term exact sequences coming from the above crossed products.
\end{proof}
\qed
\begin{theorem}\label{$K$-groups Ideals}
	Let $n,k \in \mathbb{N}$ such that $1 \leq k\leq n$. Let $J_{k}$ be the ideal of $C_{\Theta}^{0,n}$ generated by the projection  $1-\overrightarrow{\prod_{i=1}^{k}}s_{i}^{0,n}\overrightarrow{\prod_{i=1}^{k}}(s_{i}^{0,n})^*$
	Then the following hold.
	\begin{enumerate}[(i)]
		\item $K_{0}(J_{k})=\mathbb{Z}^{2^{k-1}}$ generated by $\{[\mathbb{P}_i]: 1 \leq i \leq 2^{k-1}\}$.
		\item $K_{1}(J_{k})=\mathbb{Z}^{2^{k-1}-1}$ generated by $\{[\mathbb{U}_i]: 1 \leq i \leq 2^{k-1}-1,\,\,\textrm{for all}\,\, i,\}$, where  $[\mathbb{U}_i] \neq [1] \}$ for any $i$.
	\end{enumerate}
\end{theorem}

\begin{proof} There is the following short exact sequence.
	\[
	\quad 0 \longrightarrow J_{k} \xrightarrow{i} C_{\Theta}^{0,n}\xrightarrow{\beta}  C_{\Theta}^{k,n-k}\longrightarrow  0.
	\]
	where $\beta=\beta_{0}\circ\beta_{1}\circ\cdots\circ\beta_{k-1}$. The corresponding six term short exact sequence of $K$-groups is the following.
	\begin{center}
		\begin{tikzpicture}[node distance=1cm,auto]
			\tikzset{myptr/.style={decoration={markings,mark=at position 1 with %
						{\arrow[scale=2,>=stealth]{>}}},postaction={decorate}}}
			\node (A){$ K_0(J_{k})$};
			\node (Up)[node distance=2cm, right of=A][label=above:$i_{*} $]{};
			\node (B)[node distance=4cm, right of=A]{$K_0(C_{\Theta}^{0,n})$};
			\node (Up)[node distance=2cm, right of=B][label=above:$\beta_*$]{};
			\node (C)[node distance=5cm, right of=B]{$K_0(C_{\Theta}^{k,n-k})$};
			\node (D)[node distance=2cm, below of=C]{$K_1(J_{k})$};
			\node (E)[node distance=4cm, left of=D]{$K_1(C_{\Theta}^{0,n})$};
			\node (F)[node distance=5cm, left of=E]{$K_1(C_{\Theta}^{k,n-k})$};
			\node (Below)[node distance=2cm, left of=D][label=below:$i_{*}$]{};
			\node (Below)[node distance=2cm, left of=E][label=below:$\beta_*$]{};
			\draw[myptr](A) to (B);
			\draw[myptr](B) to (C);
			\draw[myptr](C) to node{{ $\delta$}}(D);
			\draw[myptr](D) to (E);
			\draw[myptr](E) to (F);
			\draw[myptr](F) to node{{ $\partial$}}(A);
		\end{tikzpicture}
	\end{center}
	As $\beta_{*}[1]=[1]$, it follows that the map $\beta_{*}$ is injective which implies that the range of $\beta_{*}$ and by exactness, the kernel of $\delta$ are $\mathbb{Z}$ by Theorem (\ref{$K$-groups}) as well as the map $K_{1}(i)$ maps to $0$. By Theorem (\ref{K-groups general}), the range of $\delta$ and kernel of $K_{1}(i)=\mathbb{Z}^{2^{k-1}-1}$. Using Theorem (\ref{$K$-groups}), $K_{1}(i)$ maps to $0$. Therefore $K_{1}(J_{k})=\mathbb{Z}^{2^{k-1}-1}$ implying that $\delta$ is surjective and that the map $K_{1}(\beta)=0$. Hence the kernel of $\partial$ is $0$ and by Theorem (\ref{K-groups general}), the range of $\partial$ is $\mathbb{Z}^{2^{k-1}}$. This gives that $\partial$ is an isomorphism and $K_{0}(J_{k})=\mathbb{Z}^{2^{k-1}}$. The generators of  $K_{0}(J_{k})$ are $\{[\mathbb{P}_i]: 1 \leq i \leq 2^{k-1}\}$ and those of $K_{1}(J_{k})$ are $\{[\mathbb{U}_i]: 1 \leq i \leq 2^{k-1}-1\}$, and for all $i$,  $[\mathbb{U}_i] \neq [1] \}$.
\end{proof}

\noindent In what follows we recall the notion of $K$-stability, and compute non-stable groups of $C_{\Theta}^{m,n}$ by proving its $K$-stability.

\begin{definition}
	Let $\mathcal{A}$ be a $C^{*}$-algebra. Define a multiplication on $\mathcal{A}$ by $a\star b=a+b-ab$. An element $u \in \mathcal{A}$ is called \emph{quasi-unitary} if $u\star u^{*} = u^{*}\star u = 0$. The set of all quasi-unitary elements of $\mathcal{A}$ is denoted by $\widehat{\mathcal{U}}(\mathcal{A})$. The suspension of $\mathcal{A}$ is defined to be $S\mathcal{A}=C_{0}(\mathbb{R}) \otimes \mathcal{A}$. For $n > 1$, set $S^{n}\mathcal{A}:= S(S^{n-1}\mathcal{A})$. Define $\pi_{n}(\widehat{\mathcal{U}}(\mathcal{A})) \simeq \pi_{0}(\widehat{\mathcal{U}}(S^{n}\mathcal{A}))$. The nonstable $K$-groups of a $C^{*}$ algebra $\mathcal{A}$ are defined as
	$$k_{n}(\mathcal{A}):= \pi_{n+1}(\widehat{\mathcal{U}}(\mathcal{A})),\quad \mbox{ for } n\in \mathbb{N}_{0}\cup\{-1\}.$$
	Let $m \geq 2$. Define $i_{m}:M_{m-1}(\mathcal{A}) \rightarrow M_{m}(\mathcal{A})$ by
	$$ a \rightarrow
	\begin{bmatrix}
		a & 0 \\
		0 & 0\\
	\end{bmatrix}.
	\quad
	$$
	A $C^*$-algebra
	$\mathcal{A}$ is said to be $K$-stable if $(i_{m})_{*}:k_{n}(M_{m-1}(\mathcal{A})) \rightarrow k_{n}(M_{m}(\mathcal{A}))$ is an isomorphism for all $n\in \mathbb{N}_{0}\cup\{-1\}$. Note that  if $\mathcal{A}$ is $K$-stable then $k_{n}(\mathcal{A})=K_{0}(\mathcal{A})$ if $n$ is even and $k_{n}(\mathcal{A})=K_{1}(\mathcal{A})$ if $n$ is odd.  Here  $K_{0}(\mathcal{A})$ and $K_{1}(\mathcal{A})$ are the stable $K$ groups of $\mathcal{A}$.
\end{definition}

\begin{remark} We assume the following facts throughout the paper.
	\begin{enumerate}[(i)]
		\item Every stable $C^{*}$ algebra is $K$-stable (see Proposition $2.6$ of \cite{Tho-1991aa} for its proof).
		\item For $\Theta \in \bigwedge_{n}$, the universal $C^{*}$-algebra $\mathcal{A}_{\Theta}^{n}$ is $K$-stable for every $n \in \mathbb{N}$ such that $n \geq 2$ (see \cite{Rie-1987aa} for its proof).
	\end{enumerate}
\end{remark}

\begin{proposition} {\rm(\cite{Tho-1991aa})}\label{longexactseqfivelemma}
	Assume that the following is a short exact sequence of $C^*$-algebras.
	\[
	\zeta: 0\longrightarrow \mathcal{J} \xrightarrow{\nu} \mathcal{A}\xrightarrow{\kappa}  \mathcal{B}\longrightarrow  0.
	\]
	Then the following statements are true.
	\begin{enumerate}[(i)]
		\item The $K$-stability of $\mathcal{J},\mathcal{B}$ implies the $K$-stability of $\mathcal{A}$.
		\item The $K$-stability of $\mathcal{A},\mathcal{B}$ implies the $K$-stability of $\mathcal{J}$.
	\end{enumerate}
\end{proposition}

\begin{proof}
	The first result follows from \cite[Theorem 3.11]{Tho-1991aa}.
	
	For the second result, by using the short exact sequence $\zeta$, one gets  the following commuting diagram.
	\[
	\begin{tikzcd}
		\cdots k_{n+1}(\mathcal{A}) \arrow[d] \arrow[r,"k_{n+1}(\kappa)"] & k_{n+1}(\mathcal{B}) \arrow[d] \arrow[r, "\delta"] & k_{n}(\mathcal{J}) \arrow[d] \arrow[r, "k_{n}(\nu)"] & k_{n}(\mathcal{A}) \arrow[d] \arrow[r,"k_{n}(\kappa)"] & k_{n}(\mathcal{B}) \arrow[d] \cdots \\
		\cdots K_{n+1}(\mathcal{A}) \arrow[r,"K_{n+1}(\kappa)"] & K_{n+1}(\mathcal{B}) \arrow[r, "\delta"] & K_{n}(\mathcal{J}) \arrow[r, "K_{n}(\nu)"] & K_{n}(\mathcal{A}) \ar[r,"K_{n}(\kappa)"] & K_{n}(\mathcal{B}) \cdots\\
	\end{tikzcd}
	\]
	From the $K$-stability of $\mathcal{A}$ and $\mathcal{B}$, it follows that five vertical maps are isomorphisms. Now the claim follows from the Five lemma of Homology theory \cite[Lemma 3.3, Chapter I]{SM}.
\end{proof}

\begin{remark}
	Note that if $\mathcal{J},\mathcal{A}$ are $K$-stable, then $k_{n}(\mathcal{B})=K_{n}(\mathcal{B})$ for all $n\in\mathbb{N}_{0}$. But the same might not be concluded directly for $k_{-1}(\mathcal{B})$ and $K_{-1}(\mathcal{B})$.
\end{remark}
\begin{theorem}\label{Algebra $K$-stable}
	Let $m,n \in \mathbb{N}_{0}$ such that $m+n>1$. For $\Theta \in \bigwedge_{m+n}$, the $C^*$-algebra $C_{\Theta}^{m,n} $ is $K$-stable.
\end{theorem}

\begin{proof}  Note that  the $C^*$-algebra $C_{\Theta}^{m+n,0} $ is $K$-stable as it is the non-commutative $m+n$-torus. Now  using Proposition (\ref{exact-sequence}), one gets a  chain $\{\chi_{m+n-k,k}:1 \leq k \leq n\}$ of short exact sequence of $C^*$-algebras given as follows.
	\[
	\chi_{m+n-k,k}:  \quad 0\longrightarrow I_{m+n-k,k} \xrightarrow{i} C_{\Theta}^{m+n-k,k}\xrightarrow{\beta_{m+n-k}}  C_{\Theta}^{m+n-(k-1),k-1}\longrightarrow  0.
	\]
	By Proposition (\ref{twisted}), we have
	$$I_{m+n-k,k} \cong \mathcal{K}\otimes C_{\Theta_{[\widehat{m+n-k+1}]}}^{m+n-k,k-1}.$$
	This proves that $I_{m+n-k,k}$ is stable, hence $K$-stable. Now if we assume that  $C_{\Theta}^{m+n-(k-1),k-1}$ is $K$-stable then by Proposition (\ref{longexactseqfivelemma}),  it follows that $C_{\Theta}^{m+n-k,k}$ is $K$-stable.
	The claim now follows by repeatedly applying this argument to the chain $\chi_{m+n,0}, \chi_{m+n-1,1}\cdots \chi_{m+1,n-1}$ in the given order and the fact that $C_{\Theta}^{m+n,0} $  is $K$-stable.
\end{proof}

\begin{corollary} Let $m,n \in \mathbb{N}_{0}$ be such that $m+n>1$ and let $\Theta \in \bigwedge_{m+n}$.  Then one has for $j\in\mathbb{N}_{0}$, the non-stable $K$-groups of $C_{\Theta}^{m,n}$ as follows.
	\begin{IEEEeqnarray*}{rCl}
		k_{2j}\big(C_{\Theta}^{m,n}\big)	&\cong &
		\begin{cases}
			K_{0}(C_{\Theta}^{m,n}) \cong \mathbb{Z}^{2^{m-1}}& \mbox{ if } m\geq1, \cr
			K_{0}(C_{\Theta}^{0,n}) \cong \mathbb{Z} & \mbox{ if } m=0. \cr
		\end{cases} \\
		k_{2j+1}\big(C_{\Theta}^{m,n}\big) &\cong & k_{-1}\big(C_{\Theta}^{m,n}\big)\cong
		\begin{cases}
			K_{1}(C_{\Theta}^{m,n}) \cong \mathbb{Z}^{2^{m-1}}& \mbox{ if } m\geq1, \cr
			K_{1}(C_{\Theta}^{0,n}) \cong 0& \mbox{ if } m=0,\cr
		\end{cases}
	\end{IEEEeqnarray*}
	where $k_{i}\big(C_{\Theta}^{m,n}\big)$ is the $i$-th non-stable $K$-group of $C_{\Theta}^{m,n}$ for $i\in\mathbb{N}_{0}\cup\{-1\}$.
\end{corollary}

\begin{proof}  By Theorem \ref{Algebra $K$-stable}, we get $K$-stability of $C_{\Theta}^{m,n}$. As a consequence, the claim follows.\end{proof} \qed

Let $\mathcal{N}$ denote the bootstrap category of $C^*$-algebras (see page 228, \cite{Bla-1998aa} for details). Then any $C^*$-algebra in $\mathcal{N}$ satisfies the UCT. The following proposition says that $C_{\Theta}^{m,n}$ is in the category $\mathcal{N}$, hence satisfies the UCT.
\begin{theorem} \label{UCT}
	For $m,n \in \mathbb{N}_0$, the $C^*$-algebra $C_{\Theta}^{m,n}$ satisfies the UCT. Moreover, if $J_I$ is the closed ideal of $C_{\Theta}^{m,n}$  generated by the  defect projection of  $\prod_{i \in I}s_i^{m,n}$ then $J_I$ satisfies the UCT.
\end{theorem}
\begin{proof} Invoking  Theorem  $23.1.1$ in \cite{Bla-1998aa} (see page 233 in \cite{Bla-1998aa} or \cite{RosSch-1987aa}), it suffices to show that all these $C^*$-algebras are in the bootstrap category $\mathcal{N}$. First note that they are separable and nuclear  (see \cite{NarJaySur-2022aa}), Theorem 6.2). Consider the following short exact sequence $\chi_{m,n}$ defined in Proposition (\ref{exact-sequence}).
	\[
	\chi_{m,n}: \quad 0\longrightarrow I_{m,n} \xrightarrow{i} C_{\Theta}^{m,n}\xrightarrow{\beta_m}  C_{\Theta}^{m+1,n-1}\longrightarrow  0.
	\]
	Moreover, from Proposition (\ref{twisted}), we have
	\[I_{m,n}
	\cong \mathcal{K}\otimes C_{\Theta_{[\widehat{m+1}]}}^{m,n-1}.
	\]
	Assume that $C_{\Theta}^{m,n-1}$ and $C_{\Theta}^{m+1,n-1}$ are in $\mathcal{N}$. Then  $I_{m,n}$ is in  $\mathcal{N}$ as it is $KK$-equivalent to  $C_{\Theta}^{m,n-1}$. Since two of the $C^*$-algebras are in $\mathcal{N}$ then so is the third.
	Following this argument, the claim that  $\mathcal{N}$ contains $C_{\Theta}^{m,n}$ follows by considering the  chain $\chi_{m+n-1,0}, \chi_{m+n-1,0}, \cdots \chi_{m,n}$ in the given order and from the fact that $\mathcal{N}$ contains the noncommutative torus.
	
	By part $(iv)$ of Remark (\ref{remark1}), we can assume, without loss of generality, that $I=\{1,2,\cdots , m,m+1,m+2, \cdots m+k\}$ for some $1 \leq k \leq n$.   Using  the following short exact sequence
	\[
	0\longrightarrow J_I \xrightarrow{i} C_{\Theta}^{m,n}\xrightarrow{\beta_{m+k-1}\circ  \beta_{m+k-2}\circ  \cdots \circ \beta_{m}} C_{\Theta}^{m+k,n-k}\longrightarrow  0,
	\]
	we conclude that $J_I$ is in the category $\mathcal{N}$.
\end{proof}

\newsection{Free twist of  isometries}
In this section, we discuss the general form of a representation of  $B_{\Theta}^{m,n}$ and prove its $K$-stability.
A  more general question  that arises naturally is; does there exists a non $K$-stable $C^*$-algebra generated by a tuple of doubly noncommuting isometries? Equivalently, is there any non $K$-stable ideal of $B_{\Theta}^{m,n}$? Here we show that if $\Theta \in \bigwedge_{m+n}$ then  any homomorphic image of $B_{\Theta}^{m,n}$ under a representation is $K$-stable. This  will prove that every ideal  of $B_{\Theta}^{m,n}$ is $K$-stable. We start with a definition.

\begin{definition} \label{twisted-isometries}
	Fix $m,n \in \mathbb{N}_0$ with $m+n>1$.  Let $\Theta =\{\theta_{ij}\in \mathbb{R}: 1\leq i<j \leq m+n\}$. We define $B_{\Theta}^{m,n}$ to be the universal $C^*$-algebra generated by
	$s_1, s_2, \cdots s_{m+n}$ satisfying the relations;
	\begin{IEEEeqnarray}{rCll}
		s_is_j&=&e^{2\pi \textrm{i} \theta_{ij} }s_js_i,\,\,& \mbox{ if } 1\leq i<j \leq m+n; \label{R1}\\
		s_i^*s_i&=& 1,\,\, & \mbox{ if } 1\leq i\leq m+n; \label{R2}\\
		s_is_i^*&=&1, \,\, & \mbox{ if }  1\leq i \leq m. \label{R3}
	\end{IEEEeqnarray}
	We call $s_i$'s the standard generators of $B_{\Theta}^{m,n}$.  If we need to specify $m,n$, we write $s_i^{m,n}$ for these generators.
\end{definition}

\begin{remark}
	Note that
	\begin{enumerate}[(i)]
		
		\item the $C^*$-algebra $B_{\Theta}^{m,0}$  is the noncommtative $m$-torus. We write $\mathcal{A}_{\Theta}^{m}$ for $B_{\Theta}^{m,0}$.
		
		\item Following \cite{Mor-2013aa}, we call the $C^*$-algebra $B_{\Theta}^{0,n}$ the free twist of $n$ isometries.
		We denote
		$B_{\Theta}^{0,n}$ by $B_{\Theta}^{n}$.
		
		\item There is a chain of  canonical maps $\mu_l$, $1\leq l  \leq m+n$,
		\[
		B_{\Theta}^{0,m+n} \xrightarrow{ \mu_0 } B_{\Theta}^{1,m+n-1} \xrightarrow{\mu_1} B_{\Theta}^{2,m+n-2} \cdots \cdots
		\xrightarrow{\mu_{m+n-1}} B_{\Theta}^{m+n,0}.
		\]
		mapping the canonical generators of $B_{\Theta}^{m+n-l+1,l-1}$ to the canonical generators of $B_{\Theta}^{m+n-l,l}$.
	\end{enumerate}
\end{remark}

First, we will  prove that there exists exactly one maximal ideal of $B_{\Theta}^{m,n}$ similar to the case of 	$C_{\Theta}^{m,n}$. The proof is exactly along the lines of \cite{Mor-2013aa}.
\begin{lemma}\label{unity decomposition}
	Fix $\epsilon > 0$. Let $D$ be the linear span of elements $x( 1-s_{m+n}^{m,n}(s_{m+n}^{m,n})^{*}) x^{\prime}$, where $x$ and $x^{\prime}$ are $\ast$-monomials  in $s_{m+1}^{m,n},\cdots,s_{m+n}^{m,n}$. Let $J$ be an ideal in $B_{\Theta}^{m,n}$ and let $1=w+y+z$ be a decomposition of $1$ satisfying the following.
	\begin{enumerate}[(i)]
		\item $w \in D$.
		\item $y$ is an element of $J$.
		\item $z \in B_{\Theta}^{m,n}$ such that $||z|| < \epsilon$.
	\end{enumerate}
	Then there exists a $y^{\prime} \in J$ such that $||y^{\prime}-1||<\epsilon$.
\end{lemma}
\begin{proof}
	Consider an $r \in \mathbb{N}$ such that for the isometry
	$s=({s_{m+1}^{m,n}})^{r}\cdots({s_{m+n}^{m,n}}){r}$, it implies that $s^{*}w=0$. Take $y^{\prime}=s^{*}ys$. Then $1=s^{*}(w+y+z)s=y^{\prime}+s^{*}zs$. Thus it follows that $||y^{\prime}-1||=||s^{*}zs||\leq||z||<\epsilon$.
\end{proof}

A similar result is true if $1-s_{m+n}^{m,n}(s_{m+n}^{m,n})^{*}$ is replaced by $1-s_{m+i}^{m,n}(s_{m+i}^{m,n})^{*}$ for $1 \leq i \leq n$. If the same symbol $D$ is used to denote the linear span, it follows that  $\overline{D}$ is the ideal $\langle1-s_{m+i}(s_{m+i})^{*}\rangle$.

The following lemma gives a short exact sequence which will be used later in the present section.
\begin{lemma} \label{short exact sequence free}
	For $1 \leq i \leq n$, let $\mu_{m}^i: B_{\Theta}^{m,n} \rightarrow
	B_{\Theta}^{m+i,n-i}$ be the homomorphism $\mu_{m+i-1}\circ \mu_{m+i-2}\circ \cdots \circ \mu_{m}$ and  let  $J_{i}$ be the closed ideal generated by $1-(\overrightarrow{\prod_{j=1}^{m+i}}s_{j})(\overrightarrow{\prod_{j=1}^{m+i}}s_j)^*$.  Then one has the following short exact sequence of $C^*$-algebras:
	\[
	0 \longrightarrow J_{i} \longrightarrow
	B_{\Theta}^{m,n} \stackrel{\mu_{m}^{i}}{\longrightarrow} B_{\Theta}^{m+i,n-i} \longrightarrow 0.
	\]
\end{lemma}
\begin{proof}	It is enough to show that $\ker (\mu_m^i)=J_{ i}$. Since  $\mu_m^i(s_j)$ is unitary for each $1\leq j \leq m+i$, we get  $J_{i} \subset \ker (\mu_m^i)$. Now, consider
	a representation $\pi$ of  $B_{\Theta}^{m,n}$ on the Hilbert space $\mathcal{H}$ which vanishes on $J_{ i}$. Using the equation (\ref{R1}), we have
	\[
	1-\overrightarrow{\prod_{j=1}^{m+i}}s_j(\overrightarrow{\prod_{j=1}^{m+i}}s_j)^*=1-\overrightarrow{\prod_{j=m+1}^{m+i}}s_j(\overrightarrow{\prod_{j=m+1}^{m+i}}s_j)^*.
	\]
	Note that
	\[
	(\overrightarrow{\prod_{j=m+1, j \neq l}^{m+i}}s_j)^*\Big(1-\overrightarrow{\prod_{j=m+1}^{m+i}}s_j(\overrightarrow{\prod_{j=m+1}^{m+i}}s_j)^*\Big)(\overrightarrow{\prod_{j=m+1, j \neq l}^{m+i}}s_j) = 1-s_ls_l^* \in  J_{ i}.
	\]
	This implies that $\pi(1-s_ls_l^*)=0$ for $ 1\leq l \leq m+i$ as $\pi$ vanishes on $J_{ i}$. It is easy to see that  $\{\pi(s_l):1\leq l \leq m+n\}$ satisfy the defining relations (\ref{twisted-isometries}) of $B_{\Theta}^{m+i,n-i}$.
	By the universal property of  $B_{\Theta}^{m+i,n-i}$,
	we get a representation $\varsigma$ on $\mathcal{H}$  such that $\varsigma \circ \mu_m^i=\pi$. This proves that $\ker (\mu_m^i) \subset J_{ i}$, and hence the claim.
\end{proof}

\begin{proposition}
	Let $\Theta \in \bigwedge_{m+n}$. Then the ideal $J=\langle1-s_{m+1}(s_{m+1})^{*},1-s_{m+2}(s_{m+2})^{*},\cdots,1-s_{m+n}(s_{m+n})^{*}\rangle$ is the unique maximal closed ideal of $B_{\Theta}^{m,n}$.
\end{proposition}
\begin{proof} Consider the following short exact sequence
	\[
	0\longrightarrow J \xrightarrow{i} B_{\Theta}^{m,n}\xrightarrow{\mu} \mathcal{A}_{\Theta}^{m+n}\longrightarrow  0,
	\]
	where $\mu=\mu_{m+n-1}\circ  \mu_{m+n-2}\circ\cdots\circ \mu_{m}$. Since $\mathcal{A}_{\Theta}^{m+n}$ is simple, $J$ is a maximal ideal. Let $I$ be any proper, non-zero closed ideal of $B_{\Theta}^{m,n}$. It follows that either $I \subseteq J$ or $I+J=B_{\Theta}^{m,n}$. Assume that $I+J=B_{\Theta}^{m,n}$. Then $$1\in I+J=I+\langle1-s_{m+1}(s_{m+1})^{*}\rangle+\langle1-s_{m+2}(s_{m+2})^{*}\rangle+\cdots+\langle1-s_{m+n}(s_{m+n})^{*}\rangle$$
	Let $I'=I+\langle1-s_{m+1}(s_{m+1})^{*}\rangle+\cdots+\langle1-s_{m+n-1}(s_{m+n-1})^{*}\rangle$. Therefore $$1 \in I'+\langle1-s_{m+n}(s_{m+n})^{*}\rangle=I'+\overline{D}.$$
	Hence there exist elements $x \in \langle1-s_{m+n}(s_{m+n})^{*}\rangle$ and $y_{1} \in I'$ such that $1=x+y_{1}$. Let $\epsilon_{m}=\frac{1}{m}$ for $m \in \mathbb{N}$. There exists an element $w_{\epsilon_{m}}\in D$ such that $||x-w_{\epsilon_{m}}||<\epsilon_{m}$. Take $z_{\epsilon_{m}}=x-w_{\epsilon_{m}}$. Then there are elements $w_{m},z_{m}\in B_ {\Theta}^{m,n}$ such that $1=w_{m}+y_{1}+z_{m}$ and $||z_{m}||<\epsilon_{m}$. By applying the Lemma (\ref{unity decomposition}) for the ideal $I'$, there is a sequence of elements $\{y_{1,m}'\}\subset I'$ such that $y_{1,m}'\rightarrow 1$ as $m \rightarrow \infty$. Hence $1 \in I'$. Let $$I''=I+\langle1-s_{m+1}(s_{m+1})^{*}\rangle+\cdots+\langle1-s_{m+n-2}(s_{m+n-2})^{*}\rangle.$$ By similarly obtaining a decomposition for $1 \in I'=I''+\langle1-s_{m+n-1}(s_{m+n-1})^{*}\rangle$, there is a sequence $\{y_{2,m}''\}$ of elements in $I''$ such that $y_{2,m}''\rightarrow 1$ as $m \rightarrow \infty$ which implies that $1 \in I''$. By a similar procedure for ideals $I+\langle1-s_{m+1}(s_{m+1})^{*}\rangle+\cdots+\langle1-s_{m+n-i}(s_{m+n-i})^{*}\rangle$ for all $2 \leq i \leq n-1$, we conclude that $I=B_{\Theta}^{m,n}$ when $I \nsubseteq J$.
\end{proof}

\begin{proposition}
	Let $m,n \in \mathbb{N}_{0};m+n>1$ and let $\Theta \in \bigwedge_{m+n}$. Let $\Gamma=\{\gamma_{ij}\in \mathbb{R}:1 \leq i<j\leq m+n\}$. Then the following statements are true.
	\begin{enumerate}[(i)]
		\item $B_{\Theta}^{m,n} \cong B_{\Gamma}^{m,n}$ implies $\mathcal{A}_{\Theta}^{m+n} \cong \mathcal{A}_{\Gamma}^{m+n}$.
		\item $C_{\Theta}^{m,n} \cong C_{\Gamma}^{m,n}$ implies $\mathcal{A}_{\Theta}^{m+n} \cong \mathcal{A}_{\Gamma}^{m+n}$.
	\end{enumerate}
\end{proposition}

\begin{proof} Let $\alpha:B_{\Theta}^{m,n} \rightarrow B_{\Gamma}^{m,n}$ be an isomorphism. Consider canonical surjective homomomorphisms $\phi:B_{\Theta}^{m,n} \rightarrow \mathcal{A}_{\Theta}^{m+n}$ and $\psi:B_{\Gamma}^{m,n} \rightarrow \mathcal{A}_{\Gamma}^{m+n}$. Since $J=\textrm{ker}\phi$ is the maximal ideal in $B_{\Theta}^{m,n}$ and $\alpha$ is an isomorphism, $\alpha(\textrm{ker}\phi)$ is the maximal ideal in $B_{\Gamma}^{m,n}$. Therefore the ideal $\textrm{ker}\psi \subseteq\alpha(\textrm{ker}\phi)$. By the isomorphism $\alpha$, it follows that  $$\alpha(\textrm{ker}\phi)=\langle1-\alpha(s_{m+1})\alpha(s_{m+1}^{*}),1-\alpha(s_{m+2})\alpha(s_{m+2}^{*}),\cdots,1-\alpha(s_{m+n})\alpha(s_{m+n}^{*})\rangle.$$ This implies that, for every $1 \leq i \leq n$,
	
	$$\psi(1-\alpha(s_{m+i})\alpha(s_{m+i}^{*}))=1-\psi\circ\alpha(s_{m+i}s_{m+i}^{*})=0.$$
	
	Hence, for $1 \leq i \leq n$, $1-\alpha(s_{m+i})\alpha(s_{m+i}^{*}) \in \textrm{ker}\psi$ which implies $\alpha(\textrm{ker}\phi) \subseteq \textrm{ker}\psi$, i.e $\alpha(\textrm{ker}\phi) = \textrm{ker}\psi$. So, $\mathcal{A}_{\Theta}^{m+n} \cong B_{\Theta}^{m,n}/\textrm{ker}\phi=B_{\Gamma}^{m,n}/\alpha(\textrm{ker}\phi)=B_{\Gamma}^{m,n}/\textrm{ker}\psi=\mathcal{A}_{\Gamma}^{m+n}$.
\end{proof}

\noindent\textbf{Form of a representation}:  Let $\pi: B_{\Theta}^{m,n} \rightarrow \mathcal{L}(\mathcal{H})$ be a unital  representation of $B_{\Theta}^{m,n}$.
Let
$\mathcal{P}= 1-\pi(s_{1}s_{2}...s_{m+n}s_{m+n}^{*}s_{m+n-1}^{*}...s_{1}^{*})= 1-\pi(s_{m+1}s_{m+2}...s_{m+n}s_{m+n}^{*}s_{m+n-1}^{*}...s_{m+1}^{*})$ be the defect projection of the isometry $\pi(s_{1}...s_{m+n})$. Define
\begin{IEEEeqnarray*}{lCl}
	\mathcal{H}_{0} &=&\{h \in \mathcal{H}:
	\mbox{ for every } k >0 \mbox{ there exists  }  h_{k} \in \mathcal{H}\mbox{ such that } h=\pi(s_{1}...s_{m+n})^{k}h_{k}\},\\
	K&=&{\mbox{CLS}\{\pi(s_{1}...s_{m+n})^{k}\eta:  k \geq 0,\, \eta \in \mathcal{P}\mathcal{H}\}}.
\end{IEEEeqnarray*}
\begin{proposition} \label{isom of Hilbert spaces}
	The subspaces $\mathcal{H}_{0}$ and $K$ are closed linear reducing subspaces of $\mathcal{H}$  and   $\mathcal{H}_{0}^{\perp}=K$. Moreover, there exists a Hilbert space isomorphism $$K \simeq \ell^{2}(\mathbb{N}_{0}) \otimes \mathcal{P}\mathcal{H};  \quad  \pi(s_{1}...s_{m+n})^{n}\eta \mapsto e_{n} \otimes \eta.$$
\end{proposition}
\begin{proof} Let $\alpha=\pi(s_{1}\cdots s_{m+n})$.
	Let $\{\xi_{k}\}_{k\in \mathbb{N}} \subset \mathcal{H}_{0}$ and $\xi \in \mathcal{H}$ such that $(\xi_{k}) \rightarrow \xi$. Therefore for all $k,l>0$, there exist $\xi_{l,k}$ such that $ \xi_{k}=\pi(s_{1}...s_{m+n})^{l}\xi_{l,k}$. Then we have $$(1-\alpha^{l}(\alpha^{*})^{l})\xi=\lim_{k}(1-\alpha^{l}(\alpha^{*})^{l})\alpha^{l}\xi_{l,k}=0.$$ Therefore, we get  $\xi=\alpha^{l}(\alpha^{*})^{l}\xi$, which proves that  $\mathcal{H}_{0}$ is closed.
	Let $\mathcal{P}y \in \mathcal{P}\mathcal{H}$, $\alpha z \in \alpha\mathcal{H}$ for some $y,z \in \mathcal{H}$. Then it follows by the definitions of $\mathcal{P}$ and $\alpha$ that $<\mathcal{P}y,\alpha z>=0$. Hence $\mathcal{P}\mathcal{H} \perp \alpha\mathcal{H}$. Let $\xi \in \mathcal{H}_{0}$, $\eta \in \mathcal{P}\mathcal{H}$ and $k \in \mathbb{N}_{0}$. Then by the definition of $\mathcal{H}_{0}$ and the orthogonality of the subspaces $\mathcal{P}\mathcal{H}$ and $\alpha\mathcal{H}$, we have  $$<\xi,\alpha^{k}\eta>=<\alpha^{k+1}\eta_{k},\alpha^{k}\eta>=<\alpha\eta_{k},\eta>=0.$$
	Therefore $\xi \in K^{\perp}$, and hence $\mathcal{H}_{0} \subset K^{\perp}$.
	
	\noindent
	For the converse, let $\xi \in K^{\perp}$. Since $\alpha^{k}p(\alpha^{*})^{k}\xi \in K$, one has
	$$<\mathcal{P}(\alpha^{*})^{k}\xi,\mathcal{P}(\alpha^{*})^{k}\xi>=<\alpha^{k}\mathcal{P}(\alpha^{*})^{k}\xi,\xi>=0.$$
	This proves that  $\mathcal{P}(\alpha^{*})^{k}\xi=0$. For $k \geq 0$, we get  $$\alpha^{k+1}(\alpha^{*})^{k+1}\xi-\alpha^{k}(\alpha^{*})^{k}\xi=\alpha^{k}(\alpha\alpha^{*}-1)(\alpha^{*})^{k}\xi=\alpha^{k}\mathcal{P}(\alpha^{*})^{k}=0$$
	This gives $\alpha^{k+1}(\alpha^{*})^{k+1}\xi=\alpha^{k}(\alpha^{*})^{k}\xi$. Therefore by induction it follows that $\alpha^{k}(\alpha^{*})^{k}\xi=\xi$ for all $k>0$.
	This gives for $k >0$,  $\xi=\alpha^{k}(\alpha^{*})^{k}\xi \in \mathcal{H}_{0}.$ Hence $K^{\perp} \subset \mathcal{H}_{0} $.
\end{proof}

\noindent
Let $(\eta_{i})_{i \in I}$ be an orthonormal basis for $P\mathcal{H}$. Then $(\alpha^{k}\eta_{i})_{k \in \mathbb{N}_{0},i \in i}$ is an orthonormal basis for $K$. Define a map

$$\phi:K \rightarrow \ell^{2}(\mathbb{N}_{0}) \otimes \mathcal{P}\mathcal{H}; \alpha^{n}\eta_{i} \mapsto e_{n} \otimes \eta_{i}.$$
The map $\phi$ is an onto isomorphism.
The identification $\phi K \phi^{*}=\ell^{2}(\mathbb{N}_{0}) \otimes \mathcal{P}\mathcal{H}$ will be denoted by $K \simeq \ell^{2}(\mathbb{N}_{0}) \otimes \mathcal{P}\mathcal{H}$.
\qed\\
We will now describe certain parameters on which the form of a representation of $B_{\Theta}^{m,n}$ depends.  For  that,
define
\begin{IEEEeqnarray*}{lCll}
	u_i&=& (1-s_{m+1}\cdots \hat{s_{i}}\cdots s_{m+n}s_{m+n}^{*} \cdots \hat{s_{i}^{*}}\cdots s_{m+1}^{*})
	+(1-s_{i}s_{i}^{*})s_{m+n}^{*}\cdots\hat{s_{i}}^{*}\cdots s_{m+1}^{*}, \\
	\mathcal{P}_i&=&s_{m+1}...s_{i-1}\hat{s_{i}}
	s_{i+1}...s_{m+n}(1-s_{i}s_{i}^{*})s_{m+n}^{*}...s_{i+1}^{*}\hat{s_{i}}^{*}s_{i-1}^{*}...s_{m+1}^{*}.
\end{IEEEeqnarray*}
In the above expression, for $1 \leq i \leq m+n$ , $s_{m+1}...s_{i-1}\hat{s_{i}}
s_{i+1}...s_{m+n}=s_{m+1}...s_{i-1}s_{i+1}...s_{m+n}$ and $s_{m+n}^{*}...s_{i+1}^{*}\hat{s_{i}}^{*}s_{i-1}^{*}...s_{m+1}^{*}=s_{m+n}^{*}...s_{i+1}^{*}s_{i-1}^{*}...s_{m+1}^{*}$. The following proposition can be thought of as a generalization of Proposition $4.1$ given in \cite{Mor-2013aa}, describing a general form of a representation of $B_{\Theta}^{m,n}$.

\begin{proposition} \label{form of a representation}
	With the set-up given above, the following statements are true.
	\begin{enumerate}[(i)]
		
		\item $\pi(s_{1}...s_{m+n})|_{\mathcal{H}_{0}}$ is a unitary and $\pi(s_{1}...s_{m+n})|_{K} \simeq S^{*} \otimes 1$.
		\item For $1 \leq i \leq m+n$,  $\pi(s_{i})|_{\mathcal{H}_{0}}$ is a unitary operator.
		
		\item The defect projection $\mathcal{P}$ of  $\pi(s_1s_2\cdots s_{m+n})$ commutes with the elements  $u_{i}^{\prime}=\pi(u_{i})$ and  $\mathcal{P}_{i}^{\prime}=\pi(\mathcal{P}_{i})$ for
		$1 \leq i \leq m+n$. The operators  $u_{i}^{\prime}$ and $\mathcal{P}_{i}^{\prime}$ are unitary operators and projection operators, respectively, on $\mathcal{P}\mathcal{H}$.
		\item Using the identification of $K$ with $\ell^{2}(\mathbb{N}_{0}) \otimes \mathcal{P}\mathcal{H}$, one has the following:
		\[
		\pi(s_{k})|_{K}=\begin{cases}
			\prod_{i=1}^{k-1}(\lambda_{ik}^{N})^{*} \prod_{i=k+1}^{m+n}(\lambda_{ki}^{N})^{l} \otimes {u_{k}}^{\prime}, & \mbox{ if } 1 \leq k \leq m, \cr
			\prod_{i=1}^{k-1}\overline{\lambda_{ik}}S^* \prod_{i=1}^{k-1}(\lambda_{ik}^{N})^{*} \prod_{i=k+1}^{m+n}(\lambda_{ki}^{N})^{l} \otimes {u_{k}}^{\prime}\mathcal{P}_{k}^{\prime}&\cr
			+ \prod_{i=1}^{k-1}(\lambda_{ik}^{N})^{*} \prod_{i=k+1}^{m+n}(\lambda_{ki}^{N})^{l} \otimes {u_{k}}^{\prime}(1-\mathcal{P}_{k}^{\prime}) & \mbox{ if } m+1 \leq k \leq m+n, \cr
		\end{cases}
		\]
		where $\lambda_{ij}=e^{2\pi\textrm{i}\theta_{ij}}$ for $i < j$ and $\lambda^{N}(e_{n})=\lambda^{n}(e_{n})$.

	\end{enumerate}
\end{proposition}

\begin{proof}
	\begin{enumerate}[(i)]
		\item Using the definition of $\mathcal{H}_{0}$, we can see that $\alpha$ on $\mathcal{H}_{0}$ is an onto isometry. For every $n \in \mathbb{N}_0$, one has  $$\alpha(e_{n} \otimes \eta) = \alpha(\alpha^{n}\eta)=\alpha^{n+1}\eta = e_{n+1} \otimes \eta$$ which implies $\alpha|_{K} = S^* \otimes 1$.
		
		\item By the defining relations of  $B_{\Theta}^{m,n}$, we have $$\pi(s_{k})\alpha=\prod_{i=1}^{k-1}\lambda_{ik}^{-1}\prod_{i=k+1}^{m+n}\lambda_{k,i}\alpha\pi(s_{k})$$ for all $1 \leq k \leq m$ implying that $\pi(s_{k})$ and $\alpha$ commute upto a scalar. Hence we get  $$\pi(s_{k})\mathcal{H}_{0} \subset \mathcal{H}_{0}\,\, \textrm{for}\,\, 1 \leq k \leq m+n.$$ Let $\xi \in \mathcal{H}_{0}$. Then  for every $l \in \mathbb{N}_0$, we have
		$$\alpha^{l+1}\xi_{l+1}=\xi=\alpha\xi_{1}\,\,\mbox{ for some }\xi_{l+1},\xi_{1}\in \mathcal{H}.$$
		This gives $\xi_{1}=\alpha^{l}\xi_{l+1}$. By above, $\xi=\pi(s_{1})(\pi(s_{2})\cdots\pi(s_{m+n}))\xi_{1}$. Then $$\pi(s_{2})\cdots\pi(s_{m+n})\xi_{1}=\pi(s_{2})\cdots\pi(s_{m+n})\alpha^{l}\xi_{l+1}
		=C\alpha^{l} \pi(s_{2})\cdots\pi(s_{m+n})\xi_{l+1}\in\mathcal{H}_{0},$$
		where $C$ is a scalar.  This implies $\xi \in \pi(s_{1})\mathcal{H}_{0} $ which proves that $\mathcal{H}_{0} \subset \pi(s_{1})\mathcal{H}_{0}$ . In the same manner, for $2 \leq i \leq m+n$, $\mathcal{H}_{0} \subset \pi(s_{i})\mathcal{H}_{0}$. This gives that $\pi(s_{i})\mathcal{H}_{0}=\mathcal{H}_{0}$ for all $1 \leq i \leq m+n$, which proves that  each $\pi(s_{i})$ is a unitary on $\mathcal{H}_{0}$.
		
		\item Using a straightforward verification, one can see  that $\mathcal{P}$ commutes with each $u_{i}^{\prime}$ and each $\mathcal{P}_{i}^{\prime}$ for all $1 \leq i \leq m+n$. Therefore, these  elements when restricted, induce well-defined operators on $\mathcal{P}\mathcal{H}$. For $1 \leq i \leq m+n$, we have $$u_{i}^{\prime}(u_{i}^{\prime})^{*}=(u_{i}^{\prime})^{*}u_{i}^{\prime}=\pi(1-s_{m+1}s_{m+2}...s_{m+n}s_{m+n}^{*}...s_{m+1}^{*}).$$
		Therefore, $u_{i}^{\prime}$ is a unitary on $\mathcal{P}\mathcal{H}$. The other part follows from the fact that  the projection $\mathcal{P}_{i}^{\prime}$ commutes with $\mathcal{P}$.
		
		\item
		For $1 \leq k \leq m$, we have
		\begin{align*}
			&\pi(s_{k})(e_{l} \otimes \eta)\\
			&=
			\pi(s_{k})\alpha^{l}\eta\\
			&=\overline{\lambda_{1k}}^{l}\overline{\lambda_{2k}}^{l}...\overline{\lambda_{k-1 k}}^{l}\lambda_{k k+1}^{l}...\lambda_{k\,m+n}^{l}\pi(s_{1}s_{2}...s_{k-1}s_{k+1}...s_{m+n})^{l}\pi(s_{k})\eta \\
			&=  \prod_{i=1}^{k-1}\lambda_{ik}^{-l}\prod_{i=k+1}^{m+n}\lambda_{ki}^{l}(\alpha^{l}\pi(s_{k})\pi(s_{1}...s_{k-1}s_{k+1}...s_{m+n}s_{m+n}^{*}...s_{k+1}^{*}s_{k-1}^{*}...s_{1}^{*})\eta \\
			&+ \alpha^{l}\pi(s_{k})(1-\pi(s_{1}...s_{k-1}s_{k+1}...s_{m+n}s_{m+n}^{*}...s_{k+1}^{*}s_{k-1}^{*}...s_{1}^{*}))\eta)\\
			&= \prod_{i=1}^{k-1}\lambda_{ik}^{-l}\prod_{i=k+1}^{m+n}\lambda_{ki}^{l}(\prod_{i=1}^{k-1}\overline{\lambda_{ik}}\alpha^{l+1}\pi(s_{m+n}^{*}...s_{k+1}^{*}s_{k-1}^{*}...s_{1}^{*})\eta \\
			&+ \alpha^{l}\pi(s_{k}(1-(s_{1}...s_{k-1}s_{k+1}...s_{m+n}s_{m+n}^{*}...s_{k+1}^{*}s_{k-1}^{*}...s_{1}^{*}))\eta))\\
			&=   \prod_{i=1}^{k-1}\lambda_{ik}^{-l}\prod_{i=k+1}^{m+n}\lambda_{ki}^{l}(\prod_{i=1}^{k-1}\overline{\lambda_{ik}}\alpha^{l+1}\pi(u_{k}\mathcal{P}_{k})\eta + \alpha^{l}\pi(u_{k}(1-\mathcal{P}_{k}))\eta)\\
			&=   \prod_{i=1}^{k-1}\lambda_{ik}^{-l}\prod_{i=k+1}^{m+n}\lambda_{ki}^{l} ( \alpha^{l}u_{k}^{'}\eta)\\
			&=  \prod_{i=1}^{k-1}\lambda_{ik}^{-l}\prod_{i=k+1}^{m+n}\lambda_{ki}^{l}( e_{l}\otimes u_{k}^{'}\eta)\\			
			&=[ \prod_{i=1}^{k-1}(\lambda_{ik}^{N})^{*} \prod_{i=k+1}^{m+n}(\lambda_{ki}^{N})^{l} \otimes u_{k}^{'}]
			(e_{l} \otimes \eta) .
		\end{align*}

		For $m+1 \leq k < m+n$, we have
		\begin{align*}
			&\pi(s_{k})(e_{l} \otimes \eta)\\
			&=
			\pi(s_{k})\alpha^{l}\eta\\
			&=\overline{\lambda_{1k}}^{l}\overline{\lambda_{2k}}^{l}...\overline{\lambda_{k-1 k}}^{l}\lambda_{k k+1}^{n}...\lambda_{k m+n}^{l}\pi(s_{1}s_{2}...s_{k-1}s_{k+1}...s_{m+n})^{l}\pi(s_{k})\eta \\
			&=  \prod_{i=1}^{k-1}\lambda_{ik}^{-l}\prod_{i=k+1}^{m+n}\lambda_{ki}^{l}(\alpha^{n'}\pi(s_{k})\pi(s_{1}...s_{k-1}s_{k+1}...s_{m+n}s_{m+n}^{*}...s_{k+1}^{*}s_{k-1}^{*}...s_{1}^{*})\eta \\
			&+ \alpha^{l}\pi(s_{k})(1-\pi(s_{1}...s_{k-1}s_{k+1}...s_{m+n}s_{m+n}^{*}...s_{k+1}^{*}s_{k-1}^{*}...s_{1}^{*}))\eta)\\
			&= \prod_{i=1}^{k-1}\lambda_{ik}^{-l}\prod_{i=k+1}^{m+n}\lambda_{ki}^{l}(\prod_{i=1}^{k-1}\overline{\lambda_{ik}}\alpha^{l+1}\pi(s_{m+n}^{*}...s_{k+1}^{*}s_{k-1}^{*}...s_{1}^{*})\eta \\
			&+ \alpha^{l}\pi(s_{k}(1-(s_{1}...s_{k-1}s_{k+1}...s_{m+n}s_{m+n}^{*}...s_{k+1}^{*}s_{k-1}^{*}...s_{1}^{*}))\eta))\\
			&=   \prod_{i=1}^{k-1}\lambda_{ik}^{-l}\prod_{i=k+1}^{m+n}\lambda_{ki}^{l}(\prod_{i=1}^{k-1}\overline{\lambda_{ik}}\alpha^{l+1}\pi(u_{k}\mathcal{P}_{k})\eta + \alpha^{l}\pi(u_{k}(1-\mathcal{P}_{k}))\eta)\\
			&=   \prod_{i=1}^{k-1}\lambda_{ik}^{-l}\prod_{i=k+1}^{m+n}\lambda_{ki}^{l} (\prod_{i=1}^{k-1}\overline{\lambda_{ik}}\alpha^{l+1}u_{k}^{'}\mathcal{P}_{k}^{'}\eta + \alpha^{l}u_{k}^{'}(1-\mathcal{P}_{k}^{'})\eta)\\
			&=  \prod_{i=1}^{k-1}\lambda_{ik}^{-l}\prod_{i=k+1}^{m+n}\lambda_{ki}^{l}(\prod_{i=1}^{k-1}\overline{\lambda_{ik}}e_{l+1}\otimes u_{k}^{'}\mathcal{P}_{k}^{'}\eta + e_{l}\otimes u_{k}^{'}(1-\mathcal{P}_{k}^{'})\eta)[\textrm{}]\\
			&= [\prod_{i=1}^{k-1}\overline{\lambda_{ik}}S^* \prod_{i=1}^{k-1}(\lambda_{ik}^{N})^{*} \prod_{i=k+1}^{m+n}(\lambda_{ki}^{N})^{l} \otimes u_{k}^{'}\mathcal{P}_{k}^{'}\\
			&+ \prod_{i=1}^{k-1}(\lambda_{ik}^{N})^{*} \prod_{i=k+1}^{m+n}(\lambda_{ki}^{N})^{l} \otimes u_{k}^{'}(1-\mathcal{P}_{k}^{'})]
			(e_{l} \otimes \eta) .
		\end{align*}
	\end{enumerate}
\end{proof}

\begin{lemma} \label{stability of image of ideals}
	Let $J_{n}$ be the closed ideal of $B_{\Theta}^{m,n}$ generated by $1-(\overrightarrow{\prod_{j=1}^{m+n}}s_j)(\overrightarrow{\prod_{j=1}^{m+n}}s_j)^*$.  Let $\pi$ be a unital representation of   $B_{\Theta}^{m,n}$. Then  $\pi(J_{ n})$ is a stable $C^*$-algebra.
\end{lemma}
\begin{proof}  Let $\mathcal{P}$ be the defect projection of the isometry $\pi(\overrightarrow{\prod_{j=1}^{m+n}}s_j)$.  Define the $C^*$-algebra $E_n$ as follows:
	\[
	E_n=\mathcal{P}\pi(J_n)\mathcal{P}.
	\]
	To get the claim, it is enough to show that $\pi(J_n) \cong \mathcal{K} \otimes E_n$.
	From Proposition (\ref{form of a representation}),   it suffices to show that
	$$\restr{\pi}{K}(J_n)\cong \mathcal{K} \otimes \restr{\pi}{K}(E_n)$$ on the Hilbert space $K$. Identifying $K$ with
	$\ell^{2}(\mathbb{N}_{0}) \otimes \mathcal{P}\mathcal{H}$ as given in Proposition (\ref{isom of Hilbert spaces}), the operators $\pi(\overrightarrow{\prod_{j=1}^{m+n}s_j})$ and $\mathcal{P}$ can be identified with $S^* \otimes 1$ and $p \otimes 1$, respectively.  Therefore, any operator in $E_n$ can be written as $p \otimes T$ for some operator $T \in \mathcal{L}(\mathcal{P}\mathcal{H})$. Define
	\[
	E_n^{\prime}=\{T \in \mathcal{L}(\mathcal{P}\mathcal{H}): p \otimes T \in E_n \}.
	\]
	Then the map $T \mapsto p \otimes T$ gives an isomorphism between  $E_n^{\prime}$ and $E_n$. It  suffices to show that
	$$\pi(J_n)= \mathcal{K} \otimes E_n^{\prime}.$$
	Take $ T \in E_n^{\prime}$. Then $p \otimes T \in E_n \subset \pi(J_n)$. Hence we have
	\[
	(\pi(\overrightarrow{\prod_{j=1}^{m+n}}s_j))^k(p \otimes T)(\pi(\overrightarrow{\prod_{j=1}^{m+n}}s_j)^*)^k=((S^k)^* \otimes 1)(p \otimes T)(S^k \otimes 1)=p_k \otimes T \in \pi(J_n).
	\]
	This shows that $ \mathcal{K} \otimes E_n^{\prime} \subset \pi(J_n)$. To prove the other containment, take a monomial $L$ in $p \otimes 1$ and $\pi(s_i)$. By part $(iv)$ of the Proposition (\ref{form of a representation}), one can write $L$ as
	\[
	L=(A_1\otimes B_1)(p \otimes 1) (A_2\otimes B_2), \, \mbox{ for some } A_1, A_2 \in \mathcal{L}(\ell^2(\mathbb{N})) \mbox{ and } B_1, B_2 \in \mathcal{L}(\ell^2(\mathcal{P}\mathcal{H})).
	\]
	Now we have
	\[
	(p\otimes 1)L(p\otimes 1)=pA_1pA_2p\otimes B_1B_2=C p\otimes B_1B_2,
	\]
	for some constant $C$. This shows that $B_1B_2 \in E_n^{\prime}$. Also, $A_1pA_2 \in \mathcal{K}$. Hence we get
	\[
	L =A_1pA_2 \otimes B_1B_2\in \mathcal{K} \otimes E_n^{\prime}.
	\]
	This proves the claim.
\end{proof}
\qed
\begin{corollary} \label{stability of ideals}
	Let $J_{n}$ be the closed ideal  of $B_{\Theta}^{m,n}$ generated by $1-(\overrightarrow{\prod_{j=1}^{m+n}}s_j)(\overrightarrow{\prod_{j=1}^{m+n}}s_j)^*$.  Then  $J_{ n}$ is a stable $C^*$-algebra.
\end{corollary}
\begin{proof} The claim follows immediately if one takes $\pi$ to be a faithful representation of  $B_{\Theta}^{m,n}$ in Lemma (\ref{stability of image of ideals}).
\end{proof}	
\qed \\
Before proceeding to the main aim, we extract the following result about the truncation of $B_{\Theta}^{m,n}$.
\begin{proposition} \label{truncation} Let $\mathcal{P}$ be the defect projection of the isometry $\overrightarrow{\prod_{j=1}^{m+n}}s_j \in  B_{\Theta}^{m,n}$. Define $E=\mathcal{P}B_{\Theta}^{m,n}\mathcal{P}$. Then $E$ is the $C^*$-algebra generated by $\mathcal{P}u_i, \mathcal{P}\mathcal{P}_j$, $1\leq i \leq m+n$ and $m+1\leq j \leq m+n$.
\end{proposition}
\begin{proof} From part $(iii)$ of Propositon (\ref{form of a representation}), we have
	$$\mathcal{P}u_i\mathcal{P}=\mathcal{P}u_i \mbox{ and }  \mathcal{P}\mathcal{P}_j\mathcal{P}=\mathcal{P}\mathcal{P}_j.$$
	This shows that the $C^*$-algebra generated by $\mathcal{P}u_i$'s and  $\mathcal{P}\mathcal{P}_j$'s  is contained in $E$. To prove the other containment, take $\pi$ to be a faithful representation of $B_{\Theta}^{m,n}$ acting on a Hilbert space $\mathcal{H}$. Thanks to  Proposition (\ref{isom of Hilbert spaces}), we can assume that $$\mathcal{H} \cong  \mathcal{H}_0 \oplus (\ell^{2}(\mathbb{N}_{0})) \otimes \mathcal{P}\mathcal{H}, \quad \mathcal{P} \cong 0 \oplus ( p \otimes 1).$$ Take any monomial $Q(s_1,s_2,\cdots s_{m+n})$ in the generators $s_1,s_2,\cdots _{m+n}$. By part $(iv)$ of the Proposition (\ref{form of a representation}), it follows that
	$$\pi(\mathcal{P}Q(s_1,s_2,\cdots s_{m+n})\mathcal{P})\cong p \otimes Q^{\prime}(u_1^{\prime},u_2^{\prime}, \cdots U_{m+n}^{\prime}, \mathcal{P}_{m+1}^{\prime}, \cdots \mathcal{P}_{m+n}^{\prime}),$$
	where $Q^{\prime}$ is a polynomial in $u_i^{\prime}$'s and $\mathcal{P}_j^{\prime}$'s. Since $p \otimes Q^{\prime}(u_1^{\prime},u_2^{\prime}, \cdots U_{m+n}^{\prime}, \mathcal{P}_{m+1}^{\prime}, \cdots \mathcal{P}_{m+n}^{\prime})$ is image of a polynomial in the variables  $\mathcal{P}u_i\mathcal{P}=\mathcal{P}u_i$ and  $\mathcal{P}\mathcal{P}_j\mathcal{P}=\mathcal{P}\mathcal{P}_j$, we get the claim.
\end{proof}
\qed
\begin{theorem}  \label{K-stability free}
	Let $ m,n \in \mathbb{N}$ with $m+n >1$. Let $\pi$ be a unital representation of $B_{\Theta}^{m,n}$.  If $\Theta \in \bigwedge_{m+n}$ then  the $C^*$-algebra $\pi(B_{\Theta}^{m,n})$ is $K$-stable. In particular,  $B_{\Theta}^{m,n}$ is $K$-stable.
\end{theorem}
\begin{proof} Note that $\pi(J_n)$ is a closed ideal of $\pi(B_{\Theta}^{m,n})$ and the quotient $C^*$-algebra $\pi(B_{\Theta}^{m,n})/\pi(J_n)$ is generated by $\{\pi(s_i)+J_n: 1\leq i \leq m+n\}$. For $1 \leq i \leq n$, define $a_i=\pi(s_i)+J_n$.  It is not difficult to verify that
	\[
	a_i^*a_i=a_ia_i^*=1, \mbox{ and } a_ia_j=e^{2\pi i \theta_{ij}}a_ja_i.
	\]
	Since $\Theta \in \bigwedge_{m+n}$ and $m+n\geq 2$, it follows that the noncommutative torus $\mathcal{A}_{\Theta}^{m+n}$ is simple. This implies that  $\pi(B_{\Theta}^{m,n})/\pi(J_n)$ is isomorphic to $\mathcal{A}_{\Theta}^{m+n}$, and hence it is $K$-stable. By Corollary (\ref{stability of ideals}), $\pi(J_{n})$ is stable, hence $K$-stable.
	Consider  the following short exact sequence of $C^*$-algebras.
	\[
	0 \longrightarrow \pi(J_{ n}) \longrightarrow
	\pi(B_{\Theta}^{m,n}) \longrightarrow \pi(B_{\Theta}^{m,n})/\pi(J_n) \longrightarrow 0.
	\]
	Using Proposition (\ref{longexactseqfivelemma}), we get $K$-stability of $\pi(B_{\Theta}^{m,n})$.
	The rest of the claim follows  if
	one takes $\pi$ to be a faithful representation of  $B_{\Theta}^{m,n}$.
\end{proof}

\begin{corollary}  \label{K-stability free 2}
	Let $ m,n \in \mathbb{N}$ with $m+n >1$ and  $\Theta \in \bigwedge_{m+n}$.  Let $J$ be a proper closed ideal of $B_{\Theta}^{m,n}$. Then $J$ is $K$-stable.
\end{corollary}
\begin{proof} Take $\pi$ to be a unital representation of $B_{\Theta}^{m,n}$ such that $\ker(\pi)=J$.  Then we have the following short exact sequence of $C^*$-algebras:
	\[
	0 \longrightarrow J \longrightarrow B_{\Theta}^{m,n} \longrightarrow \pi(B_{\Theta}^{m,n}) \longrightarrow 0.
	\]
	Since $B_{\Theta}^{m,n}$ and $\pi(B_{\Theta}^{m,n} )$ are $K$-stable, thanks to Theorem (\ref{K-stability free}), the claim follows from Proposition (\ref{longexactseqfivelemma}).
\end{proof}

\begin{theorem}  \label{K-stability C}
	Let $ m,n \in \mathbb{N}$ with $m+n >1$. Let $\pi$ be a unital representation of   $C_{\Theta}^{m,n}$.  If $\Theta \in \bigwedge_{m+n}$ then  the $C^*$-algebra $\pi(C_{\Theta}^{m,n})$ is $K$-stable. In particular,  $C_{\Theta}^{m,n}$ is $K$-stable.
\end{theorem}
\begin{proof} Since generators of $C_{\Theta}^{m,n}$ satisfy the relation, there exists a surjective  homomorphism $\Phi$ from $B_{\Theta}^{m,n}$ to $C_{\Theta}^{m,n}$. Now $\pi \circ \Phi$ is a representation of $B_{\Theta}^{m,n}$ and the image $\pi \circ \Phi(B_{\Theta}^{m,n})$ is equal to $\pi(C_{\Theta}^{m,n})$. Now by Theorem (\ref{K-stability free}), it follows that $\pi(C_{\Theta}^{m,n})$ is $K$-stable. If we take $\pi$ to be a faithful representation of $C_{\Theta}^{m,n}$,  $K$-stability of $C_{\Theta}^{m,n}$ follows.
\end{proof}	
\qed

\begin{corollary} \label{K-stability C2}
	Let $ m,n \in \mathbb{N}$ with $m+n >1$ and  $\Theta \in \bigwedge_{m+n}$.  Let $J$ be a proper closed ideal of $C_{\Theta}^{m,n}$. Then $J$ is $K$-stable.
\end{corollary}

\begin{corollary}\label{Non Stable $K$ groups of Ideals}
	Let $k,n \in \mathbb{N}$ such that $1<k<n$ and let $\Theta \in \bigwedge_{n}$. Let $J_{k}$ be the closed ideal of $C_{\Theta}^{m,n}$ generated by $1-(\overrightarrow{\prod_{j=1}^{m+k}}s_j)(\overrightarrow{\prod_{j=1}^{m+k}}s_j)^*$. Then the following holds.
	\begin{IEEEeqnarray*}{rCll}
		&k_{2m}\big(J_{k}\big)&\cong K_{0}(J_{k})\cong\mathbb{Z}^{2^{k-1}}\\
		&k_{2m+1}\big(J_{k}\big)&\cong k_{-1}\big(J_{k}\big)\cong K_{1}(J_{k})\cong\mathbb{Z}^{2^{k-1}-1}
	\end{IEEEeqnarray*}
	where $m \in \mathbb{N}_{0}$ and $k_{i}\big(J_{k}\big)$ is the $i$-th non-stable $K$-group of $J_{k}$ for $i\in\mathbb{N}_{0}\cup\{-1\}$.
\end{corollary}

\begin{proof}
	The claim follows from Corollary (\ref{K-stability C2}) and Theorem (\ref{$K$-groups Ideals}).
\end{proof}

\begin{definition} Let $\Theta=\{\theta_{ij}\in \mathbb{R}: 1\leq i<j < \infty\}$ be an infinite tuple of real numbers. Fix $m \in \mathbb{N}_0$.  We  define $C_{\Theta}^{m,\infty}$ to be the universal $C^*$-algebra generated by
	$s_1, s_2, \cdots $ satisfying the following relations;
	\begin{IEEEeqnarray}{rCll}
		s_i^*s_j&=&e^{-2\pi \textrm{i} \theta_{ij} }s_js_i^*,\,\,& \mbox{ if } 1\leq i<j < \infty; \label{Re1}\\
		s_i^*s_i&=& 1,\,\, & \mbox{ if } 1\leq i< \infty; \label{Re2}\\
		s_is_i^*&=&1, \,\, & \mbox{ if }  1\leq i \leq m. \label{Re3}
	\end{IEEEeqnarray}
	Similarly we can define $B_{\Theta}^{m,\infty}$ as the universal $C^*$-algebra generated by
	$s_1, s_2, \cdots $ satisfying the relations $(\ref{Re2}, \ref{Re3})$ and
	\begin{IEEEeqnarray}{rCll}
		s_is_j=e^{2\pi \textrm{i} \theta_{ij} }s_js_i,\, \mbox{ for  } 1\leq i<j < \infty. \label{Re4}
	\end{IEEEeqnarray}
\end{definition}
Let $\Theta_{[l]}=\{\theta_{ij}\in \mathbb{R}: 1\leq i<j \leq l\}$. By the universal property of $C_{\Theta_{[m+n]}}^{m,n}$ and $B_{\Theta_{[m+n]}}^{m,n}$, we get the following maps.
\[
\gamma_n:C_{\Theta_{[m+n]}}^{m,n} \rightarrow C_{\Theta_{[m+n+1]}}^{m,n+1}; \quad
s_i^{m,n} \mapsto  s_i^{m,n+1}, \quad \mbox{ for } 1 \leq i \leq m+n.
\]
\[
\omega_n:B_{\Theta_{[m+n]}}^{m,n} \rightarrow B_{\Theta_{[m+n+1]}}^{m,n+1}; \quad
s_i^{m,n} \mapsto  s_i^{m,n+1}, \quad \mbox{ for } 1 \leq i \leq m+n.
\]
The following proposition says that the limits of the inductive systems $(C_{\Theta_{[m+n]}}^{m,n}, \gamma_n)$, and $(B_{\Theta_{[m+n]}}^{m,n}, \omega_n)$ are  $C_{\Theta}^{m,\infty}$, and $B_{\Theta}^{m,\infty}$, respectively.
\begin{proposition}  \label{inductive}
	Let $\Theta=\{\theta_{ij}\in \mathbb{R}: 1\leq i<j < \infty\}$.
	
	Then we have
	\[
	C_{\Theta}^{m,\infty}=\lim_{n\rightarrow \infty} \, C_{\Theta_{[m+n]}}^{m,n}, \quad \mbox{ and } \quad  B_{\Theta}^{m,\infty}=\lim_{n\rightarrow \infty} \, B_{\Theta_{[m+n]}}^{m,n}.
	\]
\end{proposition}
\begin{proof}  We will prove the first part of the claim. The other part follows from the similar argument.
	Assume that
	$$ D= \lim_{n\rightarrow \infty} \, C_{\Theta_{[m+n]}}^{m,n}$$
	and
	$$\gamma_n^{\infty}: C_{\Theta}^{m,n} \rightarrow D$$
	be the associated homomorphism for each $n \in \mathbb{N}_0$.
	Using universal property of $C_{\Theta_{[m+n]}}^{m,n}$, there exists an injective homomorphism
	\[
	\Upsilon_n:C_{\Theta_{[m+n]}}^{m,n} \rightarrow C_{\Theta}^{m,\infty}
	\]
	mapping $s_i^{m,n}$ to $s_i^{m,\infty}$ for $1\leq i \leq m+n$.
	This induces an injective homomorphism
	\[
	\Upsilon_{\infty}: D \rightarrow C_{\Theta}^{m,\infty}
	\]
	such that the following diagram commutes:
	
	\[
	\begin{tikzcd}
		C_{\Theta}^{m,n} \arrow{r}{\gamma_n^{\infty}} \arrow[swap]{dr}{\Upsilon_{n}} & D \arrow{d}{\Upsilon_{\infty}} \\
		& C_{\Theta}^{m,\infty}
	\end{tikzcd}
	\]
	\noindent Since $s_i^{m,\infty}=\Upsilon_n(s_i^{m,n})$  for $1 \leq i \leq m+n$, it follows from the diagram that for all $i \in \mathbb{N}$, the element $s_i^{m,\infty}$ is in the image of $\Upsilon_{\infty}$. This proves surjectivity of $\Upsilon_n$, and hence the claim.
\end{proof}
\qed\\
Similarly as mentioned before, for the set $\Theta=\{\theta_{ij}\in\mathbb{R}:1\leq i<j<\infty\}$, let $M_{\Theta}$ denote the associated skew-symmetric matrix with the entry $\theta_{ji}=-\theta_{ij}$ for $i<j$.
Let
$$\mbox{$\bigwedge_{\infty}$}=\{M_{\Theta}: \exists  N \mbox{ such that for all  } n>N \,\, M_{\Theta_{[n]}} \in \mbox{$\bigwedge$}_{n} \}.$$
\begin{theorem} \label{Inductive limit K group}
	Let $\Theta \in \bigwedge_{\infty}$. Then $C_{\Theta}^{m,\infty}$ and $B_{\Theta}^{m,\infty}$ are $K$-stable. Moreover, UCT holds for $C_{\Theta}^{m,\infty}$ for each $m \in \mathbb{N}_0$.
\end{theorem}
\begin{proof}  Choose $n_0\in \mathbb{N}$ such that $\Theta_{[m+n_0]}  \in \bigwedge_{m+n_0}$. Hence from Proposition (\ref{inductive}), we have
	$$C_{\Theta}^{m,\infty}=\lim_{n\geq n_0, n\rightarrow \infty} \, C_{\Theta_{[m+n]}}^{m,n}, \quad B_{\Theta}^{m,\infty}=\lim_{n\geq n_0, n\rightarrow \infty} \, B_{\Theta_{[m+n]}}^{m,n}  $$
	Moreover, from Theorem (\ref{K-stability C}, \ref{K-stability free}), the $C^*$-algebras $C_{\Theta_{[m+n]}}^{m,n}$ and  $B_{\Theta_{[m+n]}}^{m,n}$ are $K$-stable for $n\geq n_0$.
	Since $K_0$, $K_1$, and $k_l$ for $l\in \mathbb{N}\cup\{-1\}$ are continuous functors, the  first part of the claim follows.
	
	From Theorem (\ref{UCT}), if follows that $C_{\Theta}^{m,n}$ is in $\mathcal{N}$ for all $n \in \mathbb{N}_0$. Since the category $\mathcal{N}$ is closed under taking countable inductive limits, the other part of the claim follows.
\end{proof}

\begin{corollary}
	Let $\Theta \in \bigwedge_{\infty}$. Then the non-stable $K$-groups of $C_{\Theta}^{m,\infty}$ are as follows.
	\begin{IEEEeqnarray}{rCll}
		&k_{2j}\big(C_{\Theta}^{m,\infty}\big)&\cong K_{0}\big(C_{\Theta}^{m,\infty}\big)\cong
		\begin{cases}
			\mathbb{Z}^{2^{m-1}}, &\mbox{ if } m\geq1, \cr
			\mathbb{Z}, &\mbox{ if } m=0 , \cr
		\end{cases}
		\\
		&k_{2j+1}\big(C_{\Theta}^{m,\infty}\big)&\cong k_{-1}\big(C_{\Theta}^{m,\infty}\big)\cong K_{1}(C_{\Theta}^{m,\infty}\big)\cong
		\begin{cases}
			\mathbb{Z}^{2^{m-1}},  &\mbox{ if } m\geq1, \cr
			0,  & \mbox{ if } m=0 , \cr
		\end{cases}
	\end{IEEEeqnarray}
	where $j\in\mathbb{N}_{0}$.
\end{corollary}
\begin{proof} Theorem (\ref{Inductive limit K group})\, gives the $K$-stability of $C_{\Theta}^{m,\infty}$. This implies the claim.
\end{proof}

\newsection{$K$-stability of $\mathcal{U}$-twisted isometries}
In this section, we fix a $n\choose 2$-tuple $\mathcal{U}$ of commuting unitaries, and study $n$-tuples  of $\mathcal{U}$-twisted isometries and free $\mathcal{U}$-twisted isometries. We prove that if the spectrum  $\sigma(\mathcal{U})$  of the commutative $C^*$-algebra generated by $\mathcal{U}$ has no degenerate skew-symmetric matrix, then the $C^*$-algebra generated by such tuple is $K$-stable. Throughout this section, we call
$\sigma(\mathcal{U})$ the joint spectrum of $\mathcal{U}$.

\begin{definition} \cite{K}
	Let $\mathcal{A}$ be a $C^{*}$-algebra and let $X$ be a compact Haussdorff space. Let $\mathcal{Z}M(\mathcal{A})$ be the center of the multiplier algebra of $\mathcal{A}$. Then $\mathcal{A}$ is a \emph{$C(X)$-algebra} if there is a unital
	$\ast$-homomorphism $\psi:C(X)\rightarrow \mathcal{Z}M(\mathcal{A})$.
\end{definition}

For any $x\in X$, consider the following set.
$$C_{0}(X,\{x\})=\{f\in C(X):f(x)=0\}$$
The set $C_{0}(X,\{x\})$ is a closed ideal of $C(X)$ and $C_{0}(X,\{x\})\mathcal{A}$ is a closed, two-sided ideal of $\mathcal{A}$. Denote $\mathcal{A}/(C_{0}(X,\{x\})\mathcal{A})$ by $\mathcal{A}(x)$. Let $\pi_{x}:\mathcal{A}\rightarrow \mathcal{A}(x)$ be the quotient map. For any $x\in X$, the algebra $\mathcal{A}(x)$ is called a \emph{fibre} of $\mathcal{A}$ at $x$. Let $\pi_{x}(a)=a(x)$ for every $a\in\mathcal{A}$. This gives, for every $a\in\mathcal{A}$, a map $$\Gamma_{a}:X\rightarrow \mathbb{R};\,\,x\longrightarrow \|a(x)\|.$$

\begin{definition}\cite{D}
	The algebra $\mathcal{A}$ is called a continuous $C(X)$ algebra if the map $\Gamma_{a}$ is continuous for every $a\in\mathcal{A}$.
\end{definition}

\begin{definition} \label{definition of U_n}
	Fix $n >1$. Let $\mathcal{U}=\{U_{ij}\}_{1\leq  i<j\leq n}$  be a $n \choose 2$-tuple of  commuting unitaries acting on a Hilbert space $\mathcal{H}$. An $n$-tuple $\mathcal{V}=(V_1,V_2,\cdots , V_n)$ of isometries on $\mathcal{H}$ is called   $\mathcal{U}$-twisted isometries if
	\begin{IEEEeqnarray}{rCll}
		V_i U_{st}&=&U_{st}V_i, & \mbox{ for } 1\leq i \leq n, \, 1 \leq s < t\leq n, \label{center}\\
		V_i^*V_j&=&U_{ij}^*V_jV_i^*, & \mbox{ for } 1 \leq i < j \leq n. \label{doubly twisted relation}
	\end{IEEEeqnarray}
	We call an $n$-tuple $\mathcal{V}=(V_1,V_2,\cdots , V_n)$ of isometries  a free $\mathcal{U}$-twisted isometries if  instead of relation (\ref{doubly twisted relation}), the tuple  satisfies the following weaker relation:
	\begin{IEEEeqnarray}{rCll}
		V_iV_j&=&U_{ij}V_jV_i, & \mbox{ for } 1 \leq i < j \leq n. \label{free twisted relation}
	\end{IEEEeqnarray}
\end{definition}
Let $\mathcal{V}=(V_1,V_2,\cdots,V_n)$ be a tuple of $\mathcal{U}$-twisted isometries. Define $A_{\mathcal{V}}$ to be the $C^*$-subalgebra of $\mathcal{L}(\mathcal{H})$ generated by the isometries $V_1,V_2,\cdots , V_n$ and unitaries $\{U_{ij}\}_{1\leq i<j\leq n}$ in the center of the algebra $A_{\mathcal{V}}$.
Let $X$ be the joint spectrum of the commuting unitaries $\{U_{ij}\}_{1\leq  i<j\leq n}$. Using equation $(\ref{center})$, we get  a homomorphism
\[
\beta: C(X) \rightarrow Z(A_{\mathcal{V}}), \quad f(x)\mapsto f(\mathcal{U}).
\]
This map gives  $A_{\mathcal{V}}$ a $C(X)$-algebra structure. 	For $\Theta \in X$, define $I_{\Theta}$ to be the ideal of $A_{\mathcal{V}}$ generated by $\{\beta(f-f(\Theta)): f \in C(X)\}=\{f(\mathcal{U}):f \in C(X)\}$.
Let  $\pi_{\Theta}: A_{\mathcal{V}} \rightarrow A_{\mathcal{V}}/I_{\Theta}$  to be the quotient map.
Write $\pi_{\Theta}(a)$ as $[a]_{\Theta}$ for $a \in A_\mathcal{V}$. The following theorem establishes  $A_{\mathcal{V}}$  a continuous $C(X)$-algebra.
\begin{theorem} (\label{continuous algebra})
	Let $n>1$. Suppose $\mathcal{V}=(V_1,V_2,\cdots , V_n)$ is a $\mathcal{U}$-twisted isometries with respect to the twist $\mathcal{U}=\{U_{ij}\}_{1\leq  i<j\leq n}$. Then the  $C^*$-algebra
	$A_{\mathcal{V}}$ is a continuous $C(X)$-algebra.
\end{theorem}
\begin{proof}
	Let $\mathbf{S}=\mathcal{V}\cup\mathcal{U} \cup \mathcal{V}^* \cup\mathcal{U}^*$.  We denote by $\mathcal{P}(\mathbf{S})$ the set of all polynomials with elements of   $\mathbf{S}$ as variables. Note that, for any $c \in\textbf{S}$, we have $\Gamma_{c}(x)=\|\pi_{x}(c)\|=1$ for all $x\in X$.
	Hence  $\Gamma_{c}$ is continuous on $X$. Now take  $a,b \in \textbf{S}$. Let $x_{1}, x_{2} \in X$. Then we have
	\begin{IEEEeqnarray*}{rCll}
		\|\pi_{x_{1}}(ab)-\pi_{x_{2}}(ab)\|&=&\{\pi_{x_{1}}(a)\pi_{x_{1}}(b)-\pi_{x_{2}}(a)\pi_{x_{2}}(b)\}\\
		&=&\{\pi_{x_{1}}(a)\pi_{x_{1}}(b)-\pi_{x_{1}}(a)\pi_{x_{2}}(b)+\pi_{x_{1}}(a)\pi_{x_{2}}(b)-\pi_{x_{2}}(a)\pi_{x_{2}}(b)\}\\
		&\leq& \|\pi_{x_{1}}(a)\|\|\pi_{x_{1}}(b)-\pi_{x_{2}}(b)\|+\|\pi_{x_{1}}(a)-\pi_{x_{2}}(a)\|\|\pi_{x_{2}}(b)\|\\
		&=& \|\pi_{x_{1}}(b)-\pi_{x_{2}}(b)\|+\|\pi_{x_{1}}(a)-\pi_{x_{2}}(a)\|.
	\end{IEEEeqnarray*}
	The above implies that $\Gamma_{ab}$ is a continuous map on $X$. Similarly we have,
	\begin{IEEEeqnarray*}{rCll}
		\|\pi_{x_{1}}(a+b)-\pi_{x_{2}}(a+b)\|&=&\{\pi_{x_{1}}(a)+\pi_{x_{1}}(b)-\pi_{x_{2}}(a)-\pi_{x_{2}}(b)\}\\
		&\leq& \|\pi_{x_{1}}(a)-\pi_{x_{2}}(a)\|+\|\pi_{x_{1}}(b)-\pi_{x_{2}}(b)\|.
	\end{IEEEeqnarray*}
	Hence the map $\Gamma_{a+b}$ is continuous on $X$. Therefore for every polynomial $q \in  \mathcal{P}(\mathbf{S})$, the map $\Gamma_{q}$ is continuous on $X$. Let $a \in A_{\mathcal{V}}$ and $x_1, x_2 \in X$.  Fix $\epsilon>0$. Then we can choose a $q\in\mathcal{P}(\mathbf{S})$  such that $\|a-q\|<\epsilon$.
	\begin{IEEEeqnarray*}{rCll}
		\|\pi_{x_1}(a)-\pi_{x_{2}}(a)\|&=&\|\pi_{x_1}(a)-\pi_{x_{1}}(q)+\pi_{x_{1}}(q)-\pi_{x_{2}}(q)+\pi_{x_{2}}(q)-\pi_{x_{2}}(a)\| \\
		&\leq&\|\pi_{x_1}(a)-\pi_{x_{1}}(q)\|+\|\pi_{x_{1}}(q)-\pi_{x_{2}}(q)\|+\|\pi_{x_{2}}(q)-\pi_{x_{2}}(a)\|\\
		&\leq& 2\|a-q\|+ \|\pi_{x_{1}}(q)-\pi_{x_{2}}(q)\|.
	\end{IEEEeqnarray*}
	Thus, by the continuity of $\Gamma_{q}$ it follows that $\Gamma_{a}$ is continuous on $X$. This proves the claim.
\end{proof}

\begin{lemma} \label{doubly non-commuting}
	The tuple $([V_1]_{\Theta}, [V_2]_{\Theta}, \cdots , [V_n]_{\Theta})$ is a doubly non-commuting tuple of isometries with parameter $\Theta$.
\end{lemma}
\begin{proof} It is a immediate consequence of the fact that $[U_{ij}]_{\Theta}=[\theta_{ij}]_{\Theta}$
	for $1\leq i<j \leq n$.
\end{proof}

\begin{lemma} \label{fibers}
	For each $\Theta \in X \cap \bigwedge_n$, the $C^*$-algebra  $A_{\mathcal{V}_{\Theta}}$ is $K$-stable.
\end{lemma}
\begin{proof}  From Lemma (\ref{doubly non-commuting}) and the universal property of $C_{\Theta}^{n}$, it follows that there is a sujective homomorphism from $C_{\Theta}^{n}$ to $A_{\mathcal{V}_{\Theta}}$. This shows that $A_{\mathcal{V}_{\Theta}}$ is a homomorphic image of $C_{\Theta}^{n}$. Using  Theorem (\ref{K-stability C}),  we get  $K$-stability of $A_{\mathcal{V}_{\Theta}}$.
\end{proof}	

\begin{theorem} \label{Main}
	Let $n>1$ and let
	$\mathcal{U}=\{U_{ij}\}_{1\leq  i<j\leq n}$  be a $n \choose 2$-tuple of  commuting unitaries with joint spectrum $X$ acting on a Hilbert space $\mathcal{H}$.  Let  $\mathcal{V}=(V_1,V_2,\cdots , V_n)$ be a tuple of  $\mathcal{U}$-twisted isometries. If $X \subset \bigwedge_n$ then the $C^*$-algebra $A_\mathcal{V}$ generated by the  elements of  $\mathcal{V}\cup \mathcal{U}$ is $K$-stable.
\end{theorem}
\begin{proof} Note that $X$ is compact and metrizable. Since $X$ is a closed subset of $\bbbt^{n\choose 2} $, we have
	$$ \mbox{  covering dimension of } X \leq \mbox{  covering dimension of } \bbbt^n= {n \choose 2}.$$
	It follows from Theorem (\ref{continuous algebra}) and  Lemma (\ref{fibers}) that the $C^*$-algebra $A_\mathcal{V}$ is a continuous $C(X)$-algebra with  $K$-stable fibers.  Hence the claim  follows from  the main result of  (\cite{SetVai-2020aa}).
\end{proof}
\begin{theorem} \label{main2}
	Let $n>1$ and let
	$\mathcal{U}=\{U_{ij}\}_{1\leq  i<j\leq n}$  be a $n \choose 2$-tuple of  commuting unitaries with joint spectrum $X$ acting on a Hilbert space $\mathcal{H}$.  Let  $\mathcal{V}=(V_1,V_2,\cdots , V_n)$ be a tuple of  free $\mathcal{U}$-twisted isometries. If $X \subset \bigwedge_n$ then the $C^*$-algebra $B_\mathcal{V}$ generated by the  elements of  $\mathcal{V}\cup \mathcal{U}$   is $K$-stable.
\end{theorem}
\begin{proof}  The proof is exactly along the lines of Theorem (\ref{Main}). Using similar calculations as done in Theorem (\ref{continuous algebra}) that $B_\mathcal{V}$ is a continuous $C(X)$-algebra.  Moreover, the fibers are homomorphic image of $B_{\Theta}^{n}$, where $\Theta \in X \subset \bigwedge_n$, hence $K$-stable. Applying main result of (\cite{SetVai-2020aa}), we get the claim.
\end{proof}	

\begin{remark}
	Let $\Omega_{n}=\{\Theta=(\theta_{ij}: 1\leq i <j \leq n):\mathcal{A}_{\Theta} \mbox{ is not } K\mbox{-stable} \}$.  If the joint spectrum $\sigma(\mathcal{U})$ contains an isolated point $\theta \in \Omega_n$, then the $C^*$-algebra $A_{\mathcal{V}}$ as defined above can be written as a direct sum, one component of which is a non $K$-stable $C^*$-algebra $\mathcal{A}_{\Theta}$. Since the nonstable $K$-groups and the natural inclusion of a $C^*$-algebra $A$ into its matrix algebra $M_n(A)$ respect the direct sum decomposition, one concludes  that $A_{\mathcal{V}}$ is non $K$-stable.
\end{remark}

\newsection{Concluding remarks}
\begin{remark}  In conclusion, we would like to say the following.
	\begin{enumerate}
		\item  Even though  $K$-stability of $B_{\Theta}^{n}$ has been established in this article,  its nonstable $K$-groups are  still  unknown.  The reason is that the $K$-groups of $B_{\Theta}^{n}$ are not known for $n>2$. To compute these groups, one can proceed along the lines of \cite{Mor-2013aa}. However, the main obstacle is  that we do not have a clear understanding about the $C^*$-algebra  $E$ defined in Proposition (\ref{truncation}). As we have shown,  $E$ is generated by a set of projections and unitaries, but to compute its $K$-groups, one needs to have more information about its structure.  We will take up this problem in another article.
		\item In order to prove $K$-stability of the $C^*$-algebras $A_{\mathcal{V}}$ and $B_{\mathcal{V}}$, we impose the condition on the joint spectrum to be a subset of $\bigwedge_n$. The reason for that is the fibers may not be $K$-stable otherwise, and one can not apply the main theorem of (\cite{SetVai-2020aa}). It would be interesting to explore   the $K$-stability  for the general case.
		\item In Proposition (\ref{form of a representation}), we have described  a general form of a representation of $B_{\Theta}^{m,n}$ which depends on the image of  $u_i$'s and $\mathcal{P}_j$'s. However, if $n>2$ then it is not clear at this point of time whether these parameters are "free" or not.
		\item The $C^*$-algebras $B_{\Theta}^{m,n}$ and $C_{\Theta}^{m,n}$ have natural $\mathbb{Z}^{n}$ and
		$\bbbt^n$ action. It would be intersting to investigate along the line of Connes (\cite{Con-1994aa}) and construct "good" equivariant spectral triples on these noncommutative spaces.
	\end{enumerate}
\end{remark}

\section*{Acknowledgement}
We would like to thank Professor Prahlad Vaidyanathan
for  useful discussions on various topics.
Bipul Saurabh   acknowledges the support from  NBHM grant 02011/19/2021/NBHM(R.P)/R\&D II/8758.

\bigskip

\bigskip
\noindent{\sc Shreema  Subhash Bhatt} (\texttt{shreemab@iitgn.ac.in})\\
{\footnotesize Indian Institute of Technology, Gandhinagar,\\  Palaj, Gandhinagar 382055, India}
\bigskip

\noindent{\sc Bipul Saurabh} (\texttt{bipul.saurabh@iitgn.ac.in},  \texttt{saurabhbipul2@gmail.com})\\
{\footnotesize Department of Mathematics,\\ Indian Institute of Technology, Gandhinagar,\\  Palaj, Gandhinagar 382055, India}

\end{document}